\documentclass[a4paper,oneside,10pt]{article}
\usepackage{amsfonts}
\usepackage{amsmath}
\usepackage{amssymb}
\usepackage{graphicx}
\usepackage{color}
\usepackage{amsthm}
\usepackage{cite}
\usepackage{enumitem}
\usepackage{hyperref}
\usepackage{bbm}
\usepackage[utf8]{inputenc}
\usepackage{array}

\newtheoremstyle{uprightstyle}
{3pt} 
{3pt} 
{\upshape} 
{} 
{\bfseries} 
{.} 
{ } 
{} 
\theoremstyle{uprightstyle} 
\newtheorem{theorem}{Theorem}[section]

\newtheorem{definition}[theorem]{Definition}

\newtheorem{lemma}[theorem]{Lemma}

\newtheorem{proposition}[theorem]{Proposition}
\newtheorem{remark}[theorem]{Remark}

\newenvironment{proof}[1][Proof]{\noindent \textbf{#1.} }{\ \ $\Box$}
\numberwithin{equation}{section} 
 
\begin{document}

\title{Backward Stochastic Volterra integral equations driven \\
by $G$-Brownian motion}
\author{Bingru Zhao \thanks{
Zhongtai Securities Institute for Financial Studies, Shandong University,
Jinan, Shandong 250100, PR China. bingruzhao@mail.sdu.edu.cn.} \and %
Mingshang Hu \thanks{%
Zhongtai Securities Institute for Financial Studies, Shandong University,
Jinan, Shandong 250100, PR China. humingshang@sdu.edu.cn. Research supported
by the National Natural Science Foundation of China (No. 12326603,
11671231). }}

\makeatletter
\renewcommand{\@date}{} 
\makeatother
\maketitle
\textbf{Abstract}. In this paper, we study the Backward stochastic Volterra
integral equation driven by $G$-Brownian motion ($G$-BSVIE). By adopting a
different backward iteration method, we construct the approximating
sequences on each local interval. With the help of $G$-stochastic analysis
techniques and the monotone convergence theorem, the existence, uniqueness,
and continuity of the solution over the entire interval are established.
Moreover, we derive the comparison theorem.



\addcontentsline{toc}{section}{\hspace*{1.8em}Abstract}

\section{Introduction}

\noindent

Let $\left( \Omega ,\mathcal{F},\mathbb{F},\mathbb{P}\right) $ be a complete
probability space and $B$ be a standard Brownian motion, where $\mathbb{F=}$ 
$\mathbb{F}^{B}.$ Motivated by stochastic optimal control problems for
controlled Volterra-type systems, Yong \cite%
{yong2006Wellposedness,yong2008Wellposedness} introduced a classical
backward stochastic Volterra integral equation (BSVIE, for short) of the
following form:%
\begin{equation}
Y(t)=\psi
(t)+\int_{t}^{T}g(t,s,Y(s),Z(t,s),Z(s,t))ds-\int_{t}^{T}Z(t,s)dB_{s},
\label{BSVIE}
\end{equation}%
where $\psi :\left[ 0,T\right] \times \Omega \rightarrow \mathbb{R}$ and $%
g:\Delta \left[ 0,T\right] \times \Omega \times \mathbb{R\times R\rightarrow
R}$ are measurable mappings with $\Delta \left[ 0,T\right] =\{\left(
t,s\right) \in \left[ 0,T\right] ^{2}$ $|$ $t\leq s\}$. It is a natural
extension of the Backward stochastic differential equation (BSDE). The
theory of BSVIE has been extensively studied. In particular, when$\ g$ is
independent of $Z(s,t)$ and $\psi $ is independent of $t,$ the BSVIE (\ref%
{BSVIE}) was initially studied by Lin \cite{Lin2002adapted}, while this type
of equation driven by a jump process was discussed by Wang and Zhang \cite%
{Wangzhang2007}. Moreover, Wang and Yong \cite{WangT2015Comparison}
established the multi-dimensional comparison theorem. Hu and $\emptyset $%
ksendal \cite{HuY2019linear} provided an explicit solution to the linear
BSVIE. For more developments, see \cite%
{Ren2010jump,WenJQ2020anticipated,WangH2022FeyKac,WangT2013mean,Hamaguchi2021infinite}
and the references therein. In practical applications, it has been shown
that the BSVIE theory serves as a powerful tool for the study of optimal
control, time-inconsistent stochastic recursive utility and dynamic
time-inconsistent risk measures (see \cite%
{WangH2021risk,WangT2017forward-backward,Qian2025}).

In the fields of economics and finance, it has been found that the
traditional probabilistic framework fails to fully characterize model
uncertainty. To address this, Peng \cite{peng2019nonlinear} systematically
proposed a time-consistent $G$-expectation. Building on this work, many
studies have been devoted to extending the theory of $G$-expectation. Denis
et al. \cite{Denis2011function} proved that the $G$-expectation $\mathbb{%
\hat{E}}$ can be represented as a family of linear expectations over a
weakly compact subset $\mathcal{P}$. Hu et al. \cite{Hu2014a,Hu2014b} proved
the well-posedness of $G$-BSDE driven by $G$-Brownian motion, establishing a
connection to nonlinear PDEs. For more related works, the readers can refer
to \cite%
{LiuG2020multi,Luop2015diffusion,LiH2021reflection,Jiangl2023stable,Hu2016Qc}%
. Moreover, $G$-expectation theory provides a new research framework for
stochastic control problems, and this has attracted the attention of many
researchers. In addition, by a different method, Soner et al. \cite%
{Soner2012Bsde2} obtained a deep result of the existence and uniqueness
theorem for a new type of fully nonlinear BSDE, called 2BSDE. See \cite%
{Hu2016maximum,Buckdan2025meanG,Linyq2019reflect} for more details.

In this paper, we consider the backward stochastic Volterra integral
equation driven by $G$-Brownian motion ($G$-BSVIE) of the following type: 
\begin{equation}
Y(t)=\phi
(t)+\int_{t}^{T}f(t,s,Y(s),Z(t,s))ds+\int_{t}^{T}g(t,s,Y(s),Z(t,s))d\left%
\langle B\right\rangle _{s}-\int_{t}^{T}Z(t,s)dB_{s}-K(t,T),
\label{G-BSVIE1}
\end{equation}%
where $\phi :\left[ 0,T\right] \times \Omega \rightarrow \mathbb{R}$ and $f,$
$g:\Delta \left[ 0,T\right] \times \Omega \times \mathbb{R}^{2}\mathbb{%
\rightarrow R}$ satisfy some assumptions to be specified below. Compared
with $G$-BSDE (see (\ref{1})), we first notice that the generators $f,$ $g$
depend on $\left( t,s\right) $ and the terminal value $\phi (t)\in
L_{G}^{\beta }\left( \Omega _{T}\right) $ for each $t\in \left[ 0,T\right] .$
Second, the solution $\left( Z,K\right) $\ of $G$-BSVIE (\ref{G-BSVIE1})
also depends on both $t$ and $s,$ which causes obvious difficulties in
studying this problem. Third, to ensure the uniqueness of Eq.(\ref{G-BSVIE1}%
), the $K$-term within the $G$-BSVIE takes the form of $K(t,T)$, rather than
the difference term $K(T)-K(t)$ adopted in $G$-BSDE.

There are many challenges in this study due to the different features of $G$%
-BSVIE. On the one hand, since the solution $Z$\ to $G$-BSVIE (\ref{G-BSVIE1}%
) depends on both $t$ and $s,$ we need to introduce a new $G$-expectation
space $H_{G}^{p}(\Delta \left( 0,T\right) )$ on $\Delta \left[ 0,T\right] $
(see Definition \ref{Z_H_G}) and investigate some of its properties with
respect to the original space $H_{G}^{p}\left( 0,T\right) $ for $p>1.$
Moreover, unlike classical BSVIEs, the elements of $G$-expectation spaces
are quasi-continuous. A key question is how to ensure that the integral $%
\int_{\cdot }^{T}f(\cdot ,s,Y(s),Z(\cdot ,s))$ $ds+\int_{\cdot }^{T}g(\cdot
,s,Y(s),Z(\cdot ,s))d\left\langle B\right\rangle _{s}-\int_{\cdot
}^{T}Z(\cdot ,s)dB_{s}$ is well-defined for each $Y(\cdot )\in M_{G}^{p}(0,T)
$ and $Z(\cdot ,\cdot )\in H_{G}^{p}(\Delta \left( 0,T\right) )$. To
overcome this issue, we introduce a local continuity condition (see (H4)),
which is weaker than the modulus of continuity condition. On the other hand,
the structure of $G$-BSVIE includes an additional decreasing $G$-martingale $%
K$, which makes the contraction mapping principle inapplicable here.
Furthermore, noting that $G$-BSVIE (\ref{G-BSVIE1}) does not satisfy%
\begin{equation*}
\begin{array}{ll}
\displaystyle Y(t)= & \displaystyle Y(T-\delta )+\int_{t}^{T-\delta
}f(t,s,Y(s),Z(t,s))ds+\int_{t}^{T-\delta }g(t,s,Y(s),Z(t,s))d\left\langle
B\right\rangle _{s} \\ 
& \displaystyle-\int_{t}^{T-\delta }Z(t,s)dB_{s}-K(t,T-\delta ),%
\end{array}%
\end{equation*}%
which indicates that $G$-BSVIE (\ref{G-BSVIE1}) does not have
time-consistency (or semigroup property). In addition, since the process $%
\left( Z,K\right) $ involves two time variables, the general induction
method is also inapplicable here. To address these issues, we adopt a new
backward iteration procedure. The core of our approach is to construct
picard approximating sequences on each local interval and subsequently
extend the solution globally. Then, combining this with the $G$-stochastic
analysis techniques, we establish the existence and uniqueness of the
solution to Eq.(\ref{G-BSVIE1}). Furthermore, the continuity of the solution
is derived from the monotone convergence theorem.

Another contribution of our paper comes from the comparison theorem for $G$%
-BSVIE (\ref{G-BSVIE1}), which is a useful tool in stochastic utility
problems. Since $G$-expectation is a nonlinear expectation, the linear
combination of $G$-martingales is no longer a $G$-martingale. Therefore, the
difference $Y_{1}-Y_{2}$ cannot be directly compared with $0$, as this would
yield an extra term $K_{1}-K_{2}$. Inspired by the method for establishing
the existence and uniqueness of solutions presented in this paper, we
construct the corresponding sequences of $G$-BSVIEs on each local interval
similarly. By means of the monotonicity conditions for $f,$ $g,$ and $\phi ,$
we derive the desired results.

The paper is organized as follows. In Section 2, we present some
preliminaries related to the $G$-expectation framework. In Section 3, we
study the existence, uniqueness, and continuity of the solution to $G$%
-BSVIE. In Section 4, the comparison theorem for $G$-BSVIEs is established.

\section{Preliminaries}

\noindent

In this section, we recall some basic notions in the framework of $G$%
-expectation, as well as the key results of backward stochastic differential
equations driven by $G$-Browninan motion ($G$-BSDEs). For more developments,
the readers can refer to \cite{peng2019nonlinear}.

\subsection{$G$-expectation}

\noindent

Denote by $\Omega _{T}=C\left( \left[ 0,T\right] ,\mathbb{R}^{d}\right) $
the space of all $\mathbb{R}^{d}$-valued continuous functions on $\left[ 0,T%
\right] $ with $w_{0}=0.$ The canonical process $B$ is defined by $%
B_{t}\left( w\right) =w_{t}$ for each $w\in \Omega _{T}$ and $t\in \left[ 0,T%
\right] $. For each $0\leq t\leq T,$ we set%
\begin{equation*}
Lip\left( \Omega _{t}\right) :=\left\{ \varphi \left(
B_{t_{1}},B_{t_{2}},...,B_{t_{N}}\right) :\varphi \in C_{b.Lip}\left( 
\mathbb{R}^{d\times N}\right) ,t_{1}<\cdot \cdot \cdot <t_{N}\leq t,N\in 
\mathbb{N}\right\} ,
\end{equation*}%
where $C_{b.Lip}\left( \mathbb{R}^{d\times N}\right) $ denotes the space of
all $\mathbb{R}^{d\times N}$-valued bounded Lipschitz functions.

Let $\mathbb{S}_{d}$ be the set of all $d\times d$ symmetric matrices. Given
a monotonic sublinear function $G:\mathbb{S}_{d}\rightarrow \mathbb{R}$,
then there exists a unique bounded, convex and closed set $\Sigma \in 
\mathbb{S}_{d}^{+}$ such that%
\begin{equation*}
G\left( A\right) :=\frac{1}{2}\underset{\gamma \in \Sigma }{\sup }\mathrm{tr}%
\left[ A\gamma \right] ,\text{ }A\in \mathbb{S}_{d},
\end{equation*}%
where $\mathbb{S}_{d}^{+}$ is the set of all $d\times d$ symmetric
positive-definite matrices$.$ In this paper, we suppose that $G$ is
non-degenerate, i.e., there are two constants $0<\underline{\sigma }^{2}\leq 
\bar{\sigma}^{2}<\infty $ such that $\gamma \geq \underline{\sigma }%
^{2}I_{d} $ for any $\gamma \in \Sigma .$

By using a nonlinear parabolic PDE, Peng \cite%
{peng2008multi,peng2019nonlinear} construct a consistent $G$-expectation
spce $( \Omega _{T},Lip\left( \Omega _{T}\right) $, $\mathbb{\hat{E}}) .$
The canonical process $B$ under $\mathbb{\hat{E}}$ is called the $G$%
-Brownian motion. For each $p\geq 1,$ let $L_{G}^{p}\left( \Omega
_{T}\right) $ be the completion of $Lip\left( \Omega _{T}\right) $ under the
norm $\left\Vert \eta \right\Vert _{L_{G}^{p}}=\mathbb{\hat{E}}\left[
\left\vert \eta \right\vert ^{p}\right] ^{1/p}.$ Furthermore, the $G$%
-expectation $\mathbb{\hat{E}}$ can be extended continuously to $%
L_{G}^{1}\left( \Omega _{T}\right) $ under the norm $\left\Vert \cdot
\right\Vert _{L_{G}^{1}}.$

The $G$-expectation has the following representation theorem.

\begin{theorem}[\negthinspace \negthinspace \protect\cite{Denis2011function}]

There exists a unique weakly compact convex collection of probability
measures $\mathcal{P}$ on $\left( \Omega _{T},\mathcal{B}\left( \Omega
_{T}\right) \right) $ such that 
\begin{equation*}
\mathbb{\hat{E}}\left[ X\right] =\underset{P\in \mathcal{P}}{\sup }E_{P}%
\left[ X\right] \text{ for each }X\in L_{G}^{1}\left( \Omega _{T}\right) ,
\end{equation*}%
where $\mathcal{B}\left( \Omega _{T}\right) =\sigma \left( B_{s}:s\leq
T\right) $.
\end{theorem}

Then we define capacity 
\begin{equation*}
c\left( A\right) =\underset{P\in \mathcal{P}}{\sup }P\left( A\right) ,\text{ 
}A\in \mathcal{B}\left( \Omega_T \right) .
\end{equation*}%
A set $A\in \mathcal{B}\left( \Omega _{T}\right) $ is polar if $c\left(
A\right) =0$. A property holds \textquotedblleft
quasi-surely\textquotedblright\ (q.s.) if it holds except for a polar set.

\begin{definition}
Let $\pi _{T}=\left\{ 0=u_{0}<u_{1}<\cdot \cdot \cdot <u_{N}=T\right\} $ be
the partition of $\left[ 0,T\right] $ for each $N\in \mathbb{N}$. Denote by $%
M_{G}^{0}(0,T)$ the set of processes on $\left[ 0,T\right] $ in the
following form:%
\begin{equation*}
\eta _{t}\left( w\right) =\overset{N-1}{\underset{j=0}{\sum }}\xi _{j}\left(
w\right) I_{\left[ u_{j},u_{j+1}\right) }\left( t\right) ,\text{ }\xi
_{j}\in Lip\left( \Omega _{u_{j}}\right) .
\end{equation*}%
For each $p\geq 1,$ let $M_{G}^{p}(0,T)$ (resp. $H_{G}^{p}(0,T)$)\ be the
completion of $M_{G}^{0}(0,T)$ under the norm $\left\Vert \eta \right\Vert
_{M_{G}^{p}}=\left( \mathbb{\hat{E}}\left[ \int_{0}^{T}\left\vert \eta
_{t}\right\vert ^{p}dt\right] \right) ^{1/p}$ (resp. $\left\Vert \eta
\right\Vert _{H_{G}^{p}}=\left( \mathbb{\hat{E}}\left[ \left(
\int_{0}^{T}\left\vert \eta _{t}\right\vert ^{2}dt\right) ^{p/2}\right]
\right) ^{1/p}$).

Denote by $S_{G}^{0}(0,T)=\left\{ h\left( t,B_{t_{1}\wedge t},B_{t_{2}\wedge
t},...,B_{t_{N}\wedge t}\right) :h\in C_{b.Lip}\left( \mathbb{R}%
^{N+1}\right) ,t_{1}<\cdot \cdot \cdot <t_{N}=T,N\in \mathbb{N}\right\} $.
For each $p\geq 1,$ let $S_{G}^{p}(0,T)$\ be the completion of $%
S_{G}^{0}(0,T)$ under the norm  $\|\eta\|_{S_G^p} = \left\{ \hat{E}\left[ \sup_{t\in [0,T]} |\eta_t|^p \right] \right\}^{1/p}.$
\end{definition}

Accordingly, the relevant \text{It}$\hat{\mathrm{o}}$'s integral $%
\int_{0}^{\cdot }\eta \left( s\right) dB_{s}$ and $\int_{0}^{\cdot }\xi
\left( s\right) d\left\langle B\right\rangle _{s}$ are well defined for each 
$\eta \in M_{G}^{2}(0,T)$ and $\xi \in M_{G}^{1}(0,T).$ Recalling
Proposition 3.4.5 and Corollary 3.5.5 in \cite{peng2019nonlinear}, we have $%
\underline{\sigma }^{2}dt\leq d\left\langle B\right\rangle _{s}\leq 
\overline{\sigma }^{2}dt,$ q.s.

\begin{theorem}
For each $\xi \in H_{G}^{\alpha }(0,T)$ with $\alpha \geq 1$ and $p>0,$ we
can obtain that there exists constants $0<c_{p}<C_{p}<\infty $ such that%
\begin{equation}
\underline{\sigma }^{p}c_{p}\mathbb{\hat{E}}\left[ \left(
\int_{0}^{T}\left\vert \xi \left( s\right) \right\vert ^{2}ds\right) ^{p/2}%
\right] \leq \mathbb{\hat{E}}\left[ \underset{t\in \lbrack 0,T]}{\sup }%
\left\vert \int_{0}^{T}\xi \left( s\right) dB_{s}\right\vert ^{p}\right]
\leq \overline{\sigma }^{2}C_{p}\mathbb{\hat{E}}\left[ \left(
\int_{0}^{T}\left\vert \xi \left( s\right) \right\vert ^{2}ds\right) ^{p/2}%
\right] .  \label{BDG}
\end{equation}
\end{theorem}

The following result is the monotone convergence theorem within $G$%
-framework, which is different from the classical case.

\begin{theorem}[\negthinspace \negthinspace \protect\cite{Denis2011function}]

\label{monotone}Let $X_{n},$ $n\geq 1$ and $X$ are $\mathcal{B}\left( \Omega
_{T}\right) $-measureable. Suppose that $\left\{ X_{n}\right\} _{n\geq 1}\in
L_{G}^{1}\left( \Omega _{T}\right) $ satisfy $X_{n}\downarrow X,$ q.s.. Then
we obtain $\mathbb{\hat{E}}\left[ X_{n}\right] \downarrow \mathbb{\hat{E}}%
\left[ X\right] .$
\end{theorem}

\subsection{Backward Stochastic differential equations driven by $G$%
-Brownian motion}

\noindent

In the following, we review some important results and estimates for $G$%
-BSDEs. Consider the following $G$-BSDE for $d=1$: 
\begin{equation}
Y_{t}=\xi
+\int\nolimits_{t}^{T}f(s,Y_{s},Z_{s})ds-\int%
\nolimits_{t}^{T}Z_{s}dB_{s}-(K_{T}-K_{t}),  \label{1}
\end{equation}%
where the generator 
\begin{equation*}
f(t,\omega ,y,z):[0,T]\times \Omega _{T}\times \mathbb{R}^{2}\rightarrow 
\mathbb{R}
\end{equation*}%
satisfies the assumptions (A1)-(A2).

(A1) For any $y,z,$ $f(\cdot ,\cdot ,y,z)\in M_{G}^{\beta }(0,T)$ with $%
\beta >1;$

(A2) $|f(t,\omega ,y_{1},z_{1})-f(t,\omega ,y_{2},z_{2})|\leq
L(|y_{1}-y_{2}|+|z_{1}-z_{2}|)$ for some $L>1$ and $y_{i},z_{i}\in \mathbb{R}%
,i=1,2.$

For $p\geq 1,$ denote by $\mathfrak{S}_{G}^{p}(0,T)$ the set of the triplet
of processes $(Y,Z,K)$ such that $Y\in S_{G}^{p}(0,T),$ $Z\in H_{G}^{p}(0,T)$
and $K$ is a decreasing $G$-martingale with $K_{0}=0,$ $K_{T}\in
L_{G}^{p}(\Omega _{T}).$ In \cite{Hu2014a,Hu2014b,Hu2022degenerate}, Hu et
al. proved the existence and uniqueness of solution to the $G$-BSDE (\ref{1}%
), along with the comparison theorem. We now present some of their main
results.

\begin{proposition}[\negthinspace \negthinspace \protect\cite%
{Hu2014a,Hu2022degenerate}]
\label{Dguji1}Suppose that $\xi \in L_{G}^{\beta }(\Omega _{T})$ and $f$
satisfies (A1)-(A2) for some $\beta >1$. Let $(Y,Z,K)\in \mathfrak{S}%
_{G}^{\alpha }(0,T)$ be the solution of $G$-BSDE (\ref{1}) for any $1<\alpha
<\beta $. Then there exists a positive constant $C$ depending on $\alpha
,T,L,\underline{\sigma }$ such that 
\begin{eqnarray*}
\left\vert Y_{t}\right\vert ^{\alpha } &\leq &C\mathbb{\hat{E}}_{t}\left[
\left\vert \xi \right\vert ^{\alpha }+\left( \int_{t}^{T}\left\vert
f_{s}^{0}\right\vert ds\right) ^{\alpha }\right] , \\
\mathbb{\hat{E}}\left[ \left( \int_{0}^{T}\left\vert Z_{s}\right\vert
^{2}ds\right) ^{\alpha /2}\right] &\leq &C\left\{ \left\Vert Y\right\Vert
_{S_{G}^{\alpha }}^{\alpha }+\left\Vert Y\right\Vert _{_{S_{G}^{\alpha }}}^{%
\frac{\alpha }{2}}\times \left\Vert \int_{0}^{T}\left\vert
f_{s}^{0}\right\vert ds\right\Vert _{L_{G}^{\alpha }}^{\frac{\alpha }{2}%
}\right\} , \\
\mathbb{\hat{E}}\left[ \left\vert K_{T}\right\vert ^{\alpha }\right] &\leq
&C\left\{ \left\Vert Y\right\Vert _{S_{G}^{\alpha }}^{\alpha }+\left\Vert
\int_{0}^{T}\left\vert f_{s}^{0}\right\vert ds\right\Vert _{L_{G}^{\alpha
}}^{\alpha }\right\} ,
\end{eqnarray*}%
where $\left\vert f_{s}^{0}\right\vert =\left\vert f(s,0,0)\right\vert .$
\end{proposition}

\begin{proposition}[\negthinspace \negthinspace \protect\cite%
{Hu2014a,Hu2022degenerate}]
\label{Dguji2}Suppose that $\xi ^{i}\in L_{G}^{\beta }(\Omega _{T})$ and $%
f^{i},$ $i=1,2$ satisfy (A1)-(A2) for some $\beta >1$. For each $1<\alpha
<\beta ,$ let $(Y^{i},Z^{i},K^{i})\in \mathfrak{S}_{G}^{\alpha }(0,T)$ be
the solution of $G$-BSDE 
\begin{equation}
Y_{t}^{i}=\xi
^{i}+\int\nolimits_{t}^{T}f^{i}(s,Y_{s}^{i},Z_{s}^{i})ds-\int%
\nolimits_{t}^{T}Z_{s}^{i}dB_{s}-(K_{T}^{i}-K_{t}^{i}).  \label{2}
\end{equation}%
Set $\hat{\xi}=\xi ^{1}-\xi ^{2},$ $\hat{Y}_{t}=Y_{t}^{1}-Y_{t}^{2},$ $\hat{Z%
}_{t}=Z_{t}^{1}-Z_{t}^{2}$ and $\hat{K}_{t}=K_{t}^{1}-K_{t}^{2}$. Then there
exists a positive constant $C$ depending on $\alpha ,T,L,\underline{\sigma }$
such that 
\begin{equation*}
|\hat{Y}_{t}|^{\alpha }\leq C\mathbb{\hat{E}}_{t}\left[ |\hat{\xi}|^{\alpha
}+\left( \int_{t}^{T}|\hat{f}_{s}|ds\right) ^{\alpha }\right] ,
\end{equation*}%
\begin{equation*}
\mathbb{\hat{E}}\left[ \left( \int_{0}^{T}|\hat{Z}_{s}|^{2}ds\right)
^{\alpha /2}\right] \leq C\left\{ ||\hat{Y}||_{S_{G}^{\alpha }}^{\alpha }+||%
\hat{Y}||_{_{S_{G}^{\alpha }}}^{\frac{\alpha }{2}}\times \underset{i=1}{%
\overset{2}{\sum }}\left[ ||Y^{i}||_{_{S_{G}^{\alpha }}}^{\frac{\alpha }{2}%
}+\left\Vert \int_{0}^{T}|f_{s}^{i,0}|ds\right\Vert _{L_{G}^{\alpha }}^{%
\frac{\alpha }{2}}\,\right] \right\} ,
\end{equation*}%
where $|\hat{f}_{s}|=\left\vert
f^{1}(s,Y_{s}^{2},Z_{s}^{2})-f^{2}(s,Y_{s}^{2},Z_{s}^{2})\right\vert $ and $%
|f_{s}^{i,0}|=\left\vert f^{i}(s,0,0)\right\vert .$
\end{proposition}

\begin{theorem}[\negthinspace \negthinspace \protect\cite%
{Hu2014a,Hu2022degenerate}]
\label{Djie}Suppose that $\xi \in L_{G}^{\beta }(\Omega _{T})$ and $f$
satisfies (A1)-(A2) for some $\beta >1$. Then $G$-BSDE (\ref{1}) admits a
unique solution $(Y,Z,K)\in \mathfrak{S}_{G}^{\alpha }(0,T)$ for each $%
1<\alpha <\beta .$
\end{theorem}

\begin{theorem}[\negthinspace \negthinspace \protect\cite{Hu2014b}]
\label{Dcompar}Suppose that $\xi ^{i}\in L_{G}^{\beta }(\Omega _{T})$ and $%
f^{i},$ $i=1,2$ satisfy (A1)-(A2) for some $\beta >1$. Let $%
(Y^{i},Z^{i},K^{i})\in \mathfrak{S}_{G}^{\alpha }(0,T)$ be the solution of $%
G $-BSDE (\ref{2}) corresponding to $\xi ^{i}$ and $f^{i}$ for any $1<\alpha
<\beta .$ If $\xi ^{1}\geq \xi ^{2}$ and $f^{1}\geq f^{2},$ then $Y^{1}\geq
Y^{2}.$
\end{theorem}

In what follows, we provide an estimate that will be frequently applied
throughout this paper.

\begin{theorem}[\negthinspace \negthinspace \protect\cite{Soner2011,song}]
\label{Esup}Let $\xi \in L_{G}^{\alpha }(\Omega _{T})$ for each $\alpha \geq
1$ and $\delta >0,$ we have 
\begin{equation*}
\mathbb{\hat{E}}\left[ \underset{t\in \lbrack 0,T]}{\sup }\mathbb{\hat{E}}%
_{t}\left[ \left\vert \xi \right\vert ^{\alpha }\right] \right] \leq C\left( 
\mathbb{\hat{E}}\left[ \left\vert \xi \right\vert ^{\alpha +\delta }\right]
\right) ^{\alpha /(\alpha +\delta )},
\end{equation*}

where $C$ is a constant depending on $\alpha $ and $\delta .$
\end{theorem}

\section{Backward Stochastic Volterra integral equations driven by $G$%
-Brownian motion}

\noindent

In this section, we study on the $G$-expectation space $\left( \Omega
_{T},L_{G}^{1}\left( \Omega _{T}\right) ,\mathbb{\hat{E}}\right) $ for
simplicity$,$ in which $\Omega _{T}=C\left( \left[ 0,T\right] ;\mathbb{R}%
\right) $ and $\mathbb{\hat{E}}\left[ B_{1}^{2}\right] =\bar{\sigma}^{2},$ $-%
\mathbb{\hat{E}}\left[ -B_{1}^{2}\right] =\underline{\sigma }^{2}$.
Moreover, our conclusions and the proof in this section still hold for $d>1$.

For each $a,b\in \left[ 0,T\right] $, define 
\begin{equation}
\Delta \left[ a,b\right] =\left\{ \left( t,s\right) \in \left[ a,b\right]
^{2}\text{ }|\text{ }a\leq t\leq s\leq b\right\} \text{ and }\Delta
(a,b]=\left\{ \left( t,s\right) \in \left[ a,b\right] ^{2}\text{ }|\text{ }%
a<t\leq s\leq b\right\} .  \label{triangle}
\end{equation}%
Consider the Backward Stochastic Volterra integral equations driven by $G$%
-Brownian motion of the form%
\begin{equation}
Y(t)=\phi
(t)+\int_{t}^{T}f(t,s,Y(s),Z(t,s))ds-\int_{t}^{T}Z(t,s)dB_{s}-K(t,T),
\label{G-BSVIE2}
\end{equation}%
where the generator%
\begin{equation*}
f(t,s,\omega ,y,z):\Delta \left[ 0,T\right] \times \Omega \times \mathbb{R}%
^{2}\rightarrow \mathbb{R}
\end{equation*}%
satisfies the following assumptions for some $\beta >1$:

(H1) For each $y,z\in 
\mathbb{R}
$ and $t\in \lbrack 0,T],$ $f(t,\cdot ,\cdot ,y,z)\in M_{G}^{\beta }(t,T);$

(H2) There exists a positive constant $L>1$ such that for each $\left(
y_{1},z_{1}\right) ,\left( y_{2},z_{2}\right) \in 
\mathbb{R}
^{2}$ and $\left( t,s\right) \in \Delta \left[ 0,T\right] ,$ 
\begin{equation*}
\left\vert f(t,s,y_{1},z_{1})-f(t,s,y_{2},z_{2})\right\vert \leq L\left(
\left\vert y_{1}-y_{2}\right\vert +\left\vert z_{1}-z_{2}\right\vert \right)
;
\end{equation*}

(H3) $\underset{t\in \lbrack 0,T]}{\sup }\mathbb{\hat{E}}\left[
\int_{t}^{T}\left\vert f(t,s,0,0)\right\vert ^{\beta }ds\right] <\infty $;

(H4) For each $t ,t^{\prime }\in \left[ 0,T\right] $ and each fixed $N\in 
\mathbb{N}
,$%
\begin{equation*}
\underset{t^{\prime }\rightarrow t}{\lim }\mathbb{\hat{E}}\left[ \underset{%
_{\{\left\vert y\right\vert +\left\vert z\right\vert \leq N\}}}{\sup }%
\left\vert \int_{t\vee t^{\prime }}^{T}\left\vert f(t^{\prime
},s,y,z)-f(t,s,y,z)\right\vert ds\right\vert ^{\beta }\right] =0.
\end{equation*}

The analysis for $G$-BSVIE (\ref{G-BSVIE1}) is similar to $G$-BSVIE (\ref%
{G-BSVIE2}). To simplify, we only focus on $G$-BSVIE (\ref{G-BSVIE2}) with $%
g=0.$ Noting that the term~$Z$ in Volterra-type equation~depends on~both $t$
and $s$. Consequently, it is necessary to develop new $G$-expectation spaces
on $\Delta \left[ 0,T\right] $ for this study. Herein, we present these
spaces in detail.

\begin{definition}
\label{Z_H_G}Let $\pi _{T}^{N}=\left\{ 0=u_{0}<u_{1}<\cdot \cdot \cdot
<u_{N}=T\right\} $ be the partitions of $\left[ 0,T\right] $ for each $N\in 
\mathbb{N}$. Denote by $M_{G}^{0}(\Delta \left( 0,T\right) )$ the set of
processes $\eta :\Delta \left[ 0,T\right] \times \Omega \rightarrow 
\mathbb{R}
$ such that%
\begin{equation*}
\eta (t,s)=\underset{i=0}{\overset{N-1}{\sum }}\underset{j=i}{\overset{N-1}{%
\sum }}\eta (u_{i},u_{j})I_{_{\Delta _{ij}\left[ 0,T\right] }}(t,s)\text{
and }\eta (u_{i},u_{j})\in Lip(\Omega _{u_{j}}),
\end{equation*}%
where $\Delta _{ij}\left[ 0,T\right] =\left( [u_{i},u_{i+1})\times \lbrack
u_{j},u_{j+1})\right) \cap \Delta \left[ 0,T\right] $ for each $0\leq i\leq
j\leq N.$ Let $H_{G}^{p}(\Delta \left( 0,T\right) )$ be the completion of $%
M_{G}^{0}(\Delta \left( 0,T\right) )$ under the norm $\left\Vert \eta
\right\Vert _{H_{G}^{p}}=\left( \mathbb{\hat{E}}\left[ \int_{0}^{T}\left(
\int_{t}^{T}\left\vert \eta (t,s)\right\vert ^{2}ds\right) ^{p/2}dt\right]
\right) ^{1/p}$ for each $p\geq 1.$ And we set 
\begin{equation*}
\tilde{H}_{G}^{p}(\Delta \left( 0,T\right) ):=\{X\in H_{G}^{p}(\Delta \left(
0,T\right) )\text{ }|\text{ }X(t,\cdot )\in H_{G}^{p}(t,T)\text{ for each
fixed }t\in \lbrack 0,T]\}
\end{equation*}%
and 
\begin{equation*}
\tilde{M}_{G}^{p}\left( 0,T\right) :=\{X\in M_{G}^{p}\left( 0,T\right) \text{
}|\text{ }X(t)\in L_{G}^{p}(\Omega _{t})\text{ for each fixed }t\in \lbrack
0,T]\}
\end{equation*}
\end{definition}

Let $\mathfrak{\tilde{S}}_{G}^{p}(\Delta \left( 0,T\right) )$ be the set of
the triplet of processes $(Y,Z,K)$ such that $Y\in \tilde{M}_{G}^{p}(0,T),$ $%
Z\in \tilde{H}_{G}^{p}(\Delta \left( 0,T\right) ),$ $\left( K(t,s)\right)
_{s\in \lbrack t,T]}$ is a decreasing $G$-martingale with $K(t,t)=0$ and $%
K(t,T)\in L_{G}^{p}(\Omega _{T})$ for any fixed $t\in \lbrack 0,T].$

With the relevant spaces defined above, we now introduce the definition of
continuity in these spaces.

\begin{definition}
Let $\beta >1$.\newline
(1) The process $\varphi :\left[ 0,T\right] \times \Omega \rightarrow 
\mathbb{R}
$ is said to be $L_{G}^{\beta }$-continuous if it holds that for each $t\in %
\left[ 0,T\right] ,$ $\varphi \left( t\right) \in L_{G}^{\beta }(\Omega
_{T}) $ and%
\begin{equation*}
\underset{t\rightarrow t_{0}}{\lim }\mathbb{\hat{E}}\left[ \left\vert
\varphi \left( t\right) -\varphi \left( t_{0}\right) \right\vert ^{\beta }%
\right] =0,\text{ for each }t_{0}\in \left[ 0,T\right] .
\end{equation*}%
(2) The process $\lambda :\Delta \left[ 0,T\right] \times \Omega \rightarrow 
\mathbb{R}
$ is said to be $S_{G}^{\beta }$-continuous if it holds that for each $t\in %
\left[ 0,T\right] ,$ $\lambda \left( t,\cdot \right) \in S_{G}^{\beta }(t,T)$
and 
\begin{equation*}
\underset{t\rightarrow t_{0}}{\lim }\mathbb{\hat{E}}\left[ \underset{r\in
\lbrack t\vee t_{0},T]}{\sup }\left\vert \lambda \left( t,r\right) -\lambda
\left( t_{0},r\right) \right\vert ^{\beta }\right] =0,\text{ for each }%
t_{0}\in \left[ 0,T\right] .
\end{equation*}%
(3) The process $\eta :\Delta \left[ 0,T\right] \times \Omega \rightarrow 
\mathbb{R}
$ is said to be $H_{G}^{\beta }$-continuous if it holds that for each $t\in %
\left[ 0,T\right] ,$ $\eta \left( t,\cdot \right) \in H_{G}^{\beta }(t,T)$
and%
\begin{equation*}
\underset{t\rightarrow t_{0}}{\lim }\mathbb{\hat{E}}\left[ \left(
\int_{t\vee t_{0}}^{T}\left\vert \eta \left( t,s\right) -\eta \left(
t_{0},s\right) \right\vert ^{2}ds\right) ^{\beta /2}\right] =0,\text{ for
each }t_{0}\in \left[ 0,T\right] .
\end{equation*}%
\ 
\end{definition}

Similarly, the definition of uniform continuity follows in parallel.

\begin{definition}
Let $\beta >1$.\newline
(1) The process $\varphi :\left[ 0,T\right] \times \Omega \rightarrow 
\mathbb{R}
$ is said to be $L_{G}^{\beta }$-uniformly continuous if it holds that for
each $t\in \left[ 0,T\right] ,$ $\varphi \left( t\right) \in L_{G}^{\beta
}(\Omega _{T})$ and 
\begin{equation*}
\underset{n\rightarrow \infty }{\lim }\text{ }\underset{_{\left\vert
t^{\prime }-t\right\vert <\frac{1}{n}}}{\sup }\mathbb{\hat{E}}\left[
\left\vert \varphi \left( t^{\prime }\right) -\varphi \left( t\right)
\right\vert ^{\beta }\right] =0,\text{ for each }t,t^{\prime }\in \left[ 0,T%
\right] .
\end{equation*}%
(2) The process $\lambda :\Delta \left[ 0,T\right] \times \Omega \rightarrow 
\mathbb{R}
$ is said to be $S_{G}^{\beta }$-uniformly continuous if it holds that for
each $t\in \left[ 0,T\right] ,$ $\lambda \left( t,\cdot \right) \in
S_{G}^{\beta }(t,T)$ and%
\begin{equation*}
\underset{n\rightarrow \infty }{\lim }\text{ }\underset{_{\left\vert
t^{\prime }-t\right\vert <\frac{1}{n}}}{\sup }\mathbb{\hat{E}}\left[ 
\underset{r\in \lbrack t^{\prime }\vee t,T]}{\sup }\left\vert \lambda \left(
t^{\prime },r\right) -\lambda \left( t,r\right) \right\vert ^{\beta }\right]
=0,\text{ for each }t,t^{\prime }\in \left[ 0,T\right] .
\end{equation*}%
\newline
(3) The process $\eta :\Delta \left[ 0,T\right] \times \Omega \rightarrow 
\mathbb{R}
$ is said to be $H_{G}^{\beta }$-uniformly continuous if it holds that for
each $t\in \left[ 0,T\right] ,$ $\eta \left( t,\cdot \right) \in
H_{G}^{\beta }(t,T)$ and%
\begin{equation*}
\underset{n\rightarrow \infty }{\lim }\text{ }\underset{_{\left\vert
t^{\prime }-t\right\vert <\frac{1}{n}}}{\sup }\mathbb{\hat{E}}\left[ \left(
\int_{t^{\prime }\vee t}^{T}\left\vert \eta \left( t^{\prime },s\right)
-\eta \left( t,s\right) \right\vert ^{2}ds\right) ^{\beta /2}\right] =0,%
\text{ for each }t,t^{\prime }\in \left[ 0,T\right] .
\end{equation*}
\end{definition}

\begin{definition}
Suppose that $\phi (\cdot )$ is $L_{G}^{\beta }$-continuous and $f$
satisfies (H1)-(H4) for some $\beta >1$. A triplet of processes $(Y,Z,K)$ is
said to be a solution of $G$-BSVIE (\ref{G-BSVIE2}) if the following
properties hold:

(i) $(Y,Z,K)\in \mathfrak{\tilde{S}}_{G}^{\alpha }(\Delta \left( 0,T\right)
) $ for any $1<\alpha <\beta ;$

(ii) $Y(t)=\phi
(t)+\int_{t}^{T}f(t,s,Y(s),Z(t,s))ds-\int_{t}^{T}Z(t,s)dB_{s}-K(t,T)$, for $%
t\in \left[ 0,T\right] .$
\end{definition}

\subsection{The $G$-BSVIEs with generators independent of Y}

\noindent

In order to obtain the well-posedness of $G$-BSVIE (\ref{G-BSVIE2}), we
first consider a simple version of $G$-BSVIE (\ref{G-BSVIE2}) on $[0,T]$: 
\begin{equation}
Y(t)=\phi (t)+\int_{t}^{T}f(t,s,Z(t,s))ds-\int_{t}^{T}Z(t,s)dB_{s}-K(t,T),
\label{noY}
\end{equation}%
where the generator 
\begin{equation*}
f(t,s,\omega ,z):\Delta \left[ 0,T\right] \times \Omega _{T}\times \mathbb{R}%
\rightarrow \mathbb{R}
\end{equation*}%
satisfies the assumptions (H1)-(H4) for some $\beta >1$. Since the generator 
$f$ of (\ref{noY}) is independent of $y,$ the assumptions (H1)-(H4) should
be simplified. More specially, for (H4), the supremum is taken over the set $%
\left\{ \left\vert z\right\vert \leq N\right\} .$

Then we define a family of $G$-BSDEs parameterized by each fixed $t\in
\lbrack 0,T]:$%
\begin{equation}
\lambda (t,r)=\phi (t)+\int_{r}^{T}f(t,s,\mu (t,s))ds-\int_{r}^{T}\mu
(t,s)dB_{s}-\left( \eta (t,T)-\eta (t,r)\right) ,\quad \left( t,r\right) \in
\Delta \left[ 0,T\right] ,  \label{para}
\end{equation}%
where $f$ satisfies the assumptions (H1)-(H4) for some $\beta >1.$ From
Theorem \ref{Djie} we derive that $G$-BSDEs (\ref{para}) admit a unique
solution $(\lambda (t,\cdot ),\mu (t,\cdot ),\eta (t,\cdot ))\in \mathfrak{S}%
_{G}^{\alpha }(t,T)$, $1<\alpha <\beta $ and $\eta (t,t)=0$ for each fixed $%
t\in \lbrack 0,T].$ Furthermore, the estimates in Proposition \ref{Dguji1}
and Proposition \ref{Dguji2} also hold for $G$-BSDEs (\ref{para}) for each
fixed $t\in \lbrack 0,T]$.

In order to prove $\left( \lambda (t,t)\right) _{t\in \left[ 0,T\right] }$
belongs to $M_{G}^{\alpha }(0,T)$ for each $1<\alpha <\beta ,$ we establish
the following lemmas. In what follows, we denote by $C\left( \zeta ,\theta
\right) $ a positive constant depending on parameters $\zeta ,\theta $,
which may vary from line to line.

\begin{lemma}
\label{uni_f}Let $f$ be the generator of $G$-BSVIE (\ref{G-BSVIE2}). If $f$
satisfies (H1)-(H4), then for each $t ,t^{\prime} \in \left[ 0,T\right] $
and each fixed $N\in 
\mathbb{N}
,$ 
\begin{equation}
\underset{n\rightarrow \infty }{\lim }\text{ }\underset{_{\left\vert
t^{\prime }-t\right\vert <\frac{1}{n}}}{\sup }\mathbb{\hat{E}}\Biggl[%
\underset{_{\{\left\vert y\right\vert +\left\vert z\right\vert \leq N\}}}{%
\sup }\Biggl\vert\int_{t\vee t^{\prime }}^{T}\left\vert f(t^{\prime
},s,y,z)-f(t,s,y,z)\right\vert ds\Biggr\vert^{\beta }\Biggr]=0.
\label{f_proposition}
\end{equation}%
Moreover, if $\phi (\cdot )$ is $L_{G}^{\beta }$-continuous$,$ then we have $%
\underset{t\in \left[ 0,T\right] }{\sup }\mathbb{\hat{E}}\left[ \left\vert
\phi \left( t\right) \right\vert ^{\beta }\right] <\infty $ and $\phi (\cdot
)$ is $L_{G}^{\beta }$-uniformly continuous.
\end{lemma}

\begin{proof}
Define 
\begin{equation*}
l(t,r)=\mathbb{\hat{E}}\Biggl[\underset{\{\left\vert y\right\vert
+\left\vert z\right\vert \leq N\}}{\sup }\left\vert \int_{t\vee
r}^{T}\left\vert f(t,s,y,z)-f(r,s,y,z)\right\vert ds\right\vert ^{\beta }%
\Biggr],\text{ for each }(t,r)\in \lbrack 0,T]^{2}.
\end{equation*}%
Next, we prove that $l(t,r)$ is continuous in $(t,r)$.

From (H4), we obtain that for each $(t,r)\in \lbrack 0,T]^{2},$ there exist
two sequences $\left\{ t_{n}\right\} _{n\in N},\left\{ r_{n}\right\} _{n\in
N}\subset \lbrack 0,T]$ with $\underset{n\rightarrow \infty }{\lim }t_{n}=t$
and $\underset{n\rightarrow \infty }{\lim }r_{n}=r$ such that%
\begin{equation}
\underset{n\rightarrow \infty }{\lim }l\left( t_{n},t\right) =0\text{ and }%
\underset{n\rightarrow \infty }{\lim }l\left( r_{n},r\right) =0.  \label{f_1}
\end{equation}%
By Lagrange Mean Value Theorem, there exists a constant $0<\theta <1$ such
that%
\begin{equation*}
\begin{array}{l}
\displaystyle\left\vert \int_{t_{n}\vee r_{n}}^{T}\left\vert
f(t_{n},s,y,z)-f(r_{n},s,y,z)\right\vert ds\right\vert ^{\beta }-\left\vert
\int_{t\vee r}^{T}\left\vert f(t,s,y,z)-f(r,s,y,z)\right\vert ds\right\vert
^{\beta } \\ 
\displaystyle=\beta \xi ^{\beta -1}\left( \int_{t_{n}\vee
r_{n}}^{T}\left\vert f(t_{n},s,y,z)-f(r_{n},s,y,z)\right\vert ds-\int_{t\vee
r}^{T}\left\vert f(t,s,y,z)-f(r,s,y,z)\right\vert ds\right) ,%
\end{array}%
\end{equation*}%
where $\xi =\Bigl\vert\int_{t\vee r}^{T}\left\vert
f(t,s,y,z)-f(r,s,y,z)\right\vert ds+\theta (\int_{t_{n}\vee
r_{n}}^{T}\left\vert f(t_{n},s,y,z)-f(r_{n},s,y,z)\right\vert ds-\int_{t\vee
r}^{T}|f(t,s,y,z)\Bigr.$\newline
$\Bigl.-f(r,s,y,z)|ds)\Bigr\vert$. It is easy to check that 
\begin{equation*}
\mathbb{\hat{E}}\Biggl[\underset{_{\{\left\vert y\right\vert +\left\vert
z\right\vert \leq N\}}}{\sup }|\xi |^{\beta }\Biggr]\leq C\left( \underset{%
t\in \lbrack 0,T]}{\sup }\mathbb{\hat{E}}\left[ \int_{t}^{T}\left\vert
f(t,s,0,0)\right\vert ^{\beta }ds\right] +N^{\beta }\right) \leq C,
\end{equation*}%
where $C$ is a constant depending on $L,\beta ,T,N.$ Let $m_{n}=t_{n}\vee
r_{n}$ and $m=t\vee r.$ To simplify, we set $\tilde{f}%
(t,r,s,y,z)=f(t,s,y,z)-f(r,s,y,z).$ Then by H$\ddot{\mathrm{o}}$lder's
inequalty, we have%
\begin{equation}
\begin{array}{l}
\left\vert l(t_{n},r_{n})-l(t,r)\right\vert \\ 
\leq C\mathbb{\hat{E}}\left[ \underset{\{\left\vert y\right\vert +\left\vert
z\right\vert \leq N\}}{\sup }\left\vert \xi \right\vert ^{\beta }\right]
^{(\beta -1)/\beta }\mathbb{\hat{E}}\Biggl[\underset{\{\left\vert
y\right\vert +\left\vert z\right\vert \leq N\}}{\sup }\Biggl\vert%
\displaystyle\int_{m_{n}}^{T}|\tilde{f}(t_{n},r_{n},s,y,z)|ds-\displaystyle%
\int_{m}^{T}|\tilde{f}(t,r,s,y,z)|ds\Biggr\vert^{\beta }\Biggr]^{1/\beta }
\\ 
\leq C\Biggl\{\Biggl(\mathbb{\hat{E}}\Biggl[\underset{_{\{\left\vert
y\right\vert +\left\vert z\right\vert \leq N\}}}{\sup }\left\vert %
\displaystyle\int_{m_{n}\vee m}^{T}|\tilde{f}(t_{n},r_{n},s,y,z)-\tilde{f}%
(t,r,s,y,z)|ds\right\vert ^{\beta }\Biggr]^{1/\beta }\Biggr.\Biggr. \\ 
\text{ \ \ \ }\Biggl.\Biggl.+\mathbb{\hat{E}}\Biggl[\underset{_{\{\left\vert
y\right\vert +\left\vert z\right\vert \leq N\}}}{\sup }\left\vert %
\displaystyle\int_{m_{n}}^{m_{n}\vee m}|\tilde{f}(t_{n},r_{n},s,y,z)|ds%
\right\vert ^{\beta }\Biggr]^{1/\beta }+\mathbb{\hat{E}}\Biggl[\underset{%
_{_{\{\left\vert y\right\vert +\left\vert z\right\vert \leq N\}}}}{\sup }%
\left\vert \displaystyle\int_{m}^{m_{n}\vee m}|\tilde{f}(t,r,s,y,z)|ds\right%
\vert ^{\beta }\Biggr]^{1/\beta }\Biggr)\Biggr\},%
\end{array}
\label{f_3}
\end{equation}%
where $C$ is a constant depending on $L,\beta ,T,N.$ Thus, we establish the
estimates for the above terms. From (\ref{f_1}), we have 
\begin{equation}
\underset{n\rightarrow \infty }{\lim }\mathbb{\hat{E}}\Biggl[\underset{%
_{\{\left\vert y\right\vert +\left\vert z\right\vert \leq N\}}}{\sup }%
\left\vert \int_{m_{n}\vee m}^{T}|\tilde{f}(t_{n},r_{n},s,y,z)-\tilde{f}%
(t,r,s,y,z)|ds\right\vert ^{\beta }\Biggr]\leq C\underset{n\rightarrow
\infty }{\lim }\left( l\left( t_{n},t\right) +l\left( r_{n},r\right) \right)
=0,
\end{equation}%
where $C$ depends on $\beta .$ Moreover, applying H$\ddot{\mathrm{o}}$lder's
inequalty yields that 
\begin{equation}
\begin{array}{l}
\mathbb{\hat{E}}\left[ \underset{_{_{\{\left\vert y\right\vert +\left\vert
z\right\vert \leq N\}}}}{\sup }\left\vert \displaystyle\int_{m_{n}}^{m_{n}%
\vee m}|\tilde{f}(t_{n},r_{n},s,y,z)|ds\right\vert ^{\beta }\right]
^{1/\beta } \\ 
\leq \displaystyle\left( m_{n}\vee m-m_{n}\right) ^{\frac{\beta -1}{\beta }}%
\mathbb{\hat{E}}\left[ \underset{_{\{\left\vert y\right\vert +\left\vert
z\right\vert \leq N\}}}{\sup }\int_{m_{n}}^{m_{n}\vee m}|\tilde{f}%
(t_{n},r_{n},s,y,z)|^{\beta }ds\right] ^{1/\beta } \\ 
\leq \displaystyle C\left( m_{n}\vee m-m_{n}\right) ^{\frac{\beta -1}{\beta }%
}\Biggl(\underset{t\in \lbrack 0,T]}{\sup }\mathbb{\hat{E}}\left[
\int_{t}^{T}\left\vert f(t,s,0,0)\right\vert ^{\beta }ds\right] ^{1/\beta
}+\left( m_{n}\vee m-m_{n}\right) ^{\frac{1}{\beta }}LN\Biggr),%
\end{array}
\label{f_5}
\end{equation}%
where $C$ depends on $L,\beta ,T.$ The same method applies to $\mathbb{\hat{E%
}}\left[ \underset{_{\{\left\vert y\right\vert +\left\vert z\right\vert \leq
N\}}}{\sup }\left\vert \int_{m}^{m_{n}\vee m}|\tilde{f}(t,r,s,y,z)|ds\right%
\vert ^{\beta }\right] .$ Combining with (\ref{f_3})-(\ref{f_5}), we obtain $%
\underset{n\rightarrow \infty }{\lim }\left\vert
l(t_{n},r_{n})-l(t,r)\right\vert =0$ for each fixed $(t,r)\in \left[ 0,T%
\right] ^{2}.$ Noting that the interval $[0,T]^{2}$ is compact, we derive
that $l(t,r)$ is uniformly continuous in $(t,r).$ Then for each $\varepsilon
>0,$ there exists a constant $\delta >0$\ such that for each $(t^{\prime
},r^{\prime }),(t,r)\in \left[ 0,T\right] ^{2}$ with $\left\vert t^{\prime
}-t\right\vert <\delta $ and $\left\vert r^{\prime }-r\right\vert <\delta ,$ 
$\left\vert l(t^{\prime },r^{\prime })-l(t,r)\right\vert <\varepsilon .$ In
particular, taking $r=t$, we obtain $\left\vert l(t^{\prime
},t)-l(t,t)\right\vert =\left\vert l(t^{\prime },t)\right\vert <\varepsilon
, $ where $l(t,t)=0.$ Therefore, it shows that (\ref{f_proposition}) holds.
Similarly, we can obtain the uniform continuity of $\phi \left( \cdot
\right) $. Then by Lemma 3.4 in \cite{zhao}, we have $\underset{t\in \left[
0,T\right] }{\sup }\mathbb{\hat{E}}\left[ \left\vert \phi \left( t\right)
\right\vert ^{\beta }\right] <\infty .$
\end{proof}

\begin{lemma}
\label{contin_para}Suppose that $\phi (\cdot )$ is $L_{G}^{\beta }$%
-continuous and $f$ satisfies (H1)-(H4) for some $\beta >1$. Let $(\lambda
(t,\cdot ),\mu (t,\cdot ),$\newline
$\eta (t,\cdot ))$ $\in \mathfrak{S}_{G}^{\alpha }(t,T)$ be the solutions of 
$G$-BSDEs (\ref{para}) for each fixed $t\in \lbrack 0,T]$ and $1<\alpha
<\beta $. Then 
\begin{equation}
\underset{r^{\prime }\rightarrow r}{\lim }\mathbb{\hat{E}}\left[ \left\vert
\lambda (t,r^{\prime })-\lambda (t,r)\right\vert ^{\alpha }\right] =0.
\label{contin_para_1}
\end{equation}%
Moreover, $\lambda \left( \cdot ,\cdot \right) $ is $S_{G}^{\alpha }$%
-uniformly continuous and $\eta (\cdot ,\cdot )$ is $H_{G}^{\alpha }$%
-uniformly continuous.
\end{lemma}

\begin{proof}
First of all, we prove that $\lambda \left( t,\cdot \right) $ is $%
L_{G}^{\alpha }$-continuous for each fixed $t\in \lbrack 0,T]$ (see \ref%
{contin_para_1})$.$ For each $r,$ $r^{\prime }\in \lbrack t,T]$ with $%
r^{\prime }\geq r$ and fixed $t\in \lbrack 0,T]$, we have 
\begin{equation*}
\lambda (t,r^{\prime })-\lambda (t,r)=\int_{r}^{r^{\prime }}f(t,s,\mu
(t,s))ds-\int_{r}^{r^{\prime }}\mu (t,s)dB_{s}-\left( \eta (t,r^{\prime
})-\eta (t,r)\right) .
\end{equation*}%
By (\ref{BDG}) and (H2), it follows that%
\begin{equation*}
\begin{array}{ll}
\displaystyle\mathbb{\hat{E}}\left[ \left\vert \lambda (t,r^{\prime
})-\lambda (t,r)\right\vert ^{\alpha }\right] \leq & \displaystyle3\left\{
\left( r^{\prime }-r\right) ^{\alpha -1}\mathbb{\hat{E}}\left[
\int_{t}^{T}\left\vert f(t,s,0)\right\vert ^{\alpha }ds\right] +L^{\alpha
}\left( r^{\prime }-r\right) ^{\frac{\alpha }{2}}\mathbb{\hat{E}}\Biggl[%
\left( \int_{t}^{T}\left\vert \mu (t,s)\right\vert ^{2}ds\right) ^{\alpha /2}%
\Biggr]\right. \\ 
& \displaystyle\left. +\mathbb{\hat{E}}\Biggl[\left( \int_{r}^{r^{\prime
}}\left\vert \mu (t,s)\right\vert ^{2}ds\right) ^{\alpha /2}\Biggr]+\mathbb{%
\hat{E}}\left[ \left\vert \eta (t,r^{\prime })-\eta (t,r)\right\vert
^{\alpha }\right] \right\} .%
\end{array}%
\end{equation*}
From Proposition \ref{Dguji1} and Theorem \ref{Esup}, we derive for each
fixed $t\in \left[ 0,T\right] ,$ 
\begin{equation}
\mathbb{\hat{E}}\Biggl [\left( \int_{t}^{T}\left\vert \mu (t,s)\right\vert
^{2}ds\right) ^{\alpha /2}\Biggr]\leq \mathbb{\hat{E}}\left[ \underset{u\in
\lbrack t,T]}{\sup }\left\vert \lambda (t,u)\right\vert ^{\alpha }\right] +%
\mathbb{\hat{E}}\Biggl[\left( \int_{t}^{T}\left\vert f(t,s,0)\right\vert
ds\right) ^{\alpha }\Biggr] \label{contin_para_2}
\end{equation}%
and%
\begin{equation}
\begin{array}{ll}
\displaystyle\mathbb{\hat{E}}\left[ \underset{u\in \lbrack t,T]}{\sup }%
\left\vert \lambda (t,u)\right\vert ^{\alpha }\right] & \displaystyle\leq C%
\mathbb{\hat{E}}\left[ \underset{u\in \lbrack t,T]}{\sup }\mathbb{\hat{E}}%
_{u}\left[ \left\vert \phi (t)\right\vert ^{\alpha }+\left(
\int_{t}^{T}\left\vert f(t,s,0)\right\vert ds\right) ^{\alpha }\right] %
\right] \\ 
& \displaystyle\leq C\left\{ \Biggl\{\underset{t\in \lbrack 0,T]}{\sup }%
\mathbb{\hat{E}}\left[ \left\vert \phi (t)\right\vert ^{\beta }\right] ^{%
\frac{\alpha }{\beta }}+\underset{t\in \lbrack 0,T]}{\sup }\mathbb{\hat{E}}%
\Biggl[\left\vert \int_{t}^{T}\left\vert f(t,s,0)\right\vert ds\right\vert
^{\beta }\Biggr]^{\frac{\alpha }{\beta }}\Biggr\}\right\} ,%
\end{array}
\label{contin_para_3}
\end{equation}
where $C$ depends on $\alpha ,\beta ,T,L,\underline{\sigma }.$ By (H3) and
Lemma \ref{uni_f}, we obtain from (\ref{contin_para_2}) and (\ref%
{contin_para_3}) that for each $1<\alpha <\beta ,$ 
\begin{equation}
\underset{t\in \lbrack 0,T]}{\sup }\mathbb{\hat{E}}\Biggl[\left(
\int_{t}^{T}\left\vert \mu (t,s)\right\vert ^{2}ds\right) ^{\alpha /2}\Biggr]%
+\underset{t\in \lbrack 0,T]}{\sup }\mathbb{\hat{E}}\left[ \underset{u\in
\lbrack t,T]}{\sup }\left\vert \lambda (t,u)\right\vert ^{\alpha }\right]
\leq C,  \label{contin_para_4}
\end{equation}%
where $C$ depends on $\alpha ,\beta ,T,L,\underline{\sigma }.$ Then we have 
\begin{equation*}
\left( r^{\prime }-r\right) ^{\alpha -1}\mathbb{\hat{E}}\left[
\int_{t}^{T}\left\vert f(t,s,0)\right\vert ^{\alpha }ds\right] +\left(
r^{\prime }-r\right) ^{\frac{\alpha }{2}}\mathbb{\hat{E}}\Biggl[\left(
\int_{t}^{T}\left\vert \mu (t,s)\right\vert ^{2}ds\right) ^{\alpha /2}\Biggr]%
\rightarrow 0,\text{ as }r^{\prime }\rightarrow r.
\end{equation*}%
From the monotone convergence Theorem \ref{monotone}, we can easily verify
that $\mathbb{\hat{E}}\left[ \left( \int_{r}^{r^{\prime }}\left\vert \mu
(t,s)\right\vert ^{2}ds\right) ^{\alpha /2}\right] \downarrow 0$ and $%
\mathbb{\hat{E}}\left[ \left\vert \eta (t,r^{\prime })-\eta (t,r)\right\vert
^{\alpha }\right] \downarrow 0$ as $r^{\prime }\downarrow r.$ Consequently, (%
\ref{contin_para_1}) holds.

Next, by using the similar analysis in (\ref{contin_para_3}), we derive from
Proposition \ref{Dguji2} and Theorem \ref{Esup} that for each $t^{\prime
},t\in \left[ 0,T\right] $ and $\delta >0$ with $\alpha +\delta <\beta $,%
\begin{equation*}
\begin{array}{l}
\displaystyle\underset{\left\vert t^{\prime }-t\right\vert \leq \frac{1}{n}}{%
\sup }\mathbb{\hat{E}}\left[ \underset{u\in \lbrack t\vee t^{\prime },T]}{%
\sup }\left\vert \lambda (t^{\prime },u)-\lambda (t,u)\right\vert ^{\alpha }%
\right] \\ 
\displaystyle\leq C\Biggl\{\underset{\left\vert t^{\prime }-t\right\vert
\leq \frac{1}{n}}{\sup }\mathbb{\hat{E}}\left[ \left\vert \phi (t^{\prime
})-\phi (t)\right\vert ^{\alpha +\delta }\right] ^{\frac{\alpha }{\alpha
+\delta }}+\underset{\left\vert t^{\prime }-t\right\vert \leq \frac{1}{n}}{%
\sup }\mathbb{\hat{E}}\Biggl[\left( \int_{t\vee t^{\prime }}^{T}\left\vert
f(t^{\prime },s,\mu (t,s))-f(t,s,\mu (t,s))\right\vert ds\right) ^{\alpha
+\delta }\Biggr]^{\frac{\alpha }{\alpha +\delta }}\Biggr\},%
\end{array}%
\end{equation*}%
where the constant $C$ depends on $\alpha ,\delta ,T,L,\underline{\sigma }.$
By Lemma \ref{uni_f} and H$\ddot{\mathrm{o}}$lder's inequalty$,$ $\phi $ is $%
L_{G}^{\alpha +\delta }$-uniformly continuous$.$ Then it remains to show
that the second term tends to $0$ as $n\rightarrow \infty .$ From the
assumptions (H2)-(H4), we deduce from H$\ddot{\mathrm{o}}$lder's inequalty
that 
\begin{align}
& \underset{\left\vert t^{\prime }-t\right\vert \leq \frac{1}{n}}{\sup }%
\mathbb{\hat{E}}\Biggl[\left( \int_{t\vee t^{\prime }}^{T}\left\vert
f(t^{\prime },s,\mu (t,s))-f(t,s,\mu (t,s))\right\vert ds\right) ^{\alpha
+\delta }\Biggr]^{\frac{\alpha }{\alpha +\delta }}  \notag \\
& \leq C\Biggl\{\underset{\left\vert t^{\prime }-t\right\vert \leq \frac{1}{n%
}}{\sup }\mathbb{\hat{E}}\Biggl[\Biggl(\int_{t\vee t^{\prime }}^{T}\Bigl(%
|f(t^{\prime },s,\mu (t,s))-f(t^{\prime },s,0)|+|f(t^{\prime
},s,0)-f(t,s,0)|+|f(t,s,0)-f(t,s,\mu (t,s))|\Bigr)  \notag \\
& \text{ \ \ }\times I_{\{\left\vert \mu (t,s)\right\vert >N\}}ds\Biggr)%
^{\alpha +\delta }\Biggr]^{\frac{\alpha }{\alpha +\delta }}+\underset{%
\left\vert t^{\prime }-t\right\vert \leq \frac{1}{n}}{\sup }\mathbb{\hat{E}}%
\Biggl[\Biggl(\int_{t\vee t^{\prime }}^{T}\bigl\vert f(t^{\prime },s,\mu
(t,s))-f(t,s,\mu (t,s))\bigr\vert I_{\{\left\vert \mu (t,s)\right\vert \leq
N\}}ds\Biggr)^{\beta }\Biggr]^{\frac{\alpha }{\beta }}\Biggr\}  \notag \\
& \leq C\Biggl\{2\mathbb{\hat{E}}\Biggl[\Biggl(\int_{t}^{T}\left\vert \mu
(t,s)\right\vert I_{\{\left\vert \mu (t,s)\right\vert >N\}}ds\Biggr)^{\alpha
+\delta }\Biggr]^{\frac{\alpha }{\alpha +\delta }}+\underset{\left\vert
t^{\prime }-t\right\vert \leq \frac{1}{n}}{\sup }\mathbb{\hat{E}}\Biggl[%
\left( \int_{t\vee t^{\prime }}^{T}\left\vert f(t^{\prime
},s,0)-f(t,s,0)\right\vert ds\right) ^{\beta }\Biggr]^{\frac{\alpha }{\beta }%
}\Biggr.  \notag \\
& \Biggl.\quad \ +\underset{\left\vert t^{\prime }-t\right\vert \leq \frac{1%
}{n}}{\sup }\mathbb{\hat{E}}\Biggl[\left( \int_{t\vee t^{\prime
}}^{T}\left\vert f(t^{\prime },s,\mu (t,s))-f(t,s,\mu (t,s))\right\vert
I_{\{\left\vert \mu (t,s)\right\vert \leq N\}}ds\right) ^{\beta }\Biggr]^{%
\frac{\alpha }{\beta }}\Biggr\},  \label{contin_para_i_2}
\end{align}%
where the constant $C$ depends on $\alpha ,\delta ,T,L.$ Recalling (\ref%
{contin_para_4}) and taking $\delta =\frac{\beta -\alpha }{2}$, there exists
a constant $0<\gamma =\frac{\beta -\alpha }{2\left( \beta +\alpha \right) }%
<1 $ with $\alpha <\left( 1+\gamma \right) \left( \alpha +\delta \right) =%
\frac{3\beta +\alpha }{4}<\beta $ such that%
\begin{equation}
\begin{array}{ll}
\displaystyle\mathbb{\hat{E}}\Biggl[\Biggr(\int_{t}^{T}\left\vert \mu
(t,s)\right\vert I_{\{\left\vert \mu (t,s)\right\vert >N\}}ds\Biggr)^{\alpha
+\delta }\Biggr] & \displaystyle\leq \frac{1}{N^{\gamma \left( \alpha
+\delta \right) }}\mathbb{\hat{E}}\Biggl[\Biggl(\int_{t}^{T}\left\vert \mu
(t,s)\right\vert ^{1+\gamma }ds\Biggr)^{\alpha +\delta }\Biggr] \\ 
& \displaystyle\leq \frac{C}{N^{\gamma \left( \alpha +\delta \right) }}%
\mathbb{\hat{E}}\Biggl[\Biggl(\int_{t}^{T}\left\vert \mu (t,s)\right\vert
^{2}ds\Biggr)^{\frac{^{\left( 1+\gamma \right) \left( \alpha +\delta \right)
}}{2}}\Biggr]\leq \frac{C}{N^{\gamma \left( \alpha +\delta \right) }},%
\end{array}
\label{contin_para_i_3}
\end{equation}%
where the constant $C$ depends on $\alpha ,\beta ,T,L,\underline{\sigma }.$
In conclusion, sending $n\rightarrow \infty $ and then $N\rightarrow \infty $
in (\ref{contin_para_i_2})$,$ we deduce from Lemma \ref{uni_f} that $\lambda
\left( \cdot ,\cdot \right) $ is $S_{G}^{\alpha }$-uniformly continuous.

Finally, by applying Proposition \ref{Dguji2} and (\ref{contin_para_4}) to $%
G $-BSDEs (\ref{para}), we see that%
\begin{equation*}
\begin{array}{l}
\displaystyle\mathbb{\hat{E}}\left[ \left( \int_{t\vee t^{\prime
}}^{T}\left\vert \mu (t^{\prime },s)-\mu (t,s)\right\vert ^{2}ds\right)
^{\alpha /2}\right] \\ 
\displaystyle\leq C\left\{ \mathbb{\hat{E}}\left[ \underset{u\in \lbrack
t\vee t^{\prime },T]}{\sup }\left\vert \lambda (t^{\prime },u)-\lambda
(t,u)\right\vert ^{\alpha }\right] +\mathbb{\hat{E}}\left[ \underset{u\in
\lbrack t\vee t^{\prime },T]}{\sup }\left\vert \lambda (t^{\prime
},u)-\lambda (t,u)\right\vert ^{\alpha }\right] ^{1/2}\right\} ,%
\end{array}%
\end{equation*}%
where the constant $C$ depends on $\alpha ,\beta ,T,L,\underline{\sigma }.$
Therefore, from the $S_{G}^{\alpha }$-uniform continuity of $\lambda \left(
\cdot ,\cdot \right) $, we obtain that $\eta (\cdot ,\cdot )$ is $%
H_{G}^{\alpha }$-uniformly continuous.
\end{proof}

The following lemma establishes the relationship between the space $%
H_{G}^{\alpha }(0,T)$ and $H_{G}^{\alpha }(\Delta \left( 0,T\right) ),$
which is important in dealing with the well-posedness of $G$-BSVIE (\ref%
{G-BSVIE2}).

\begin{lemma}
\label{Z_space}Suppose that $\phi (\cdot )$ is $L_{G}^{\beta }$-continuous
and $f$ satisfies (H1)-(H4) for some $\beta >1$. Let $\mu (t,\cdot )\in
H_{G}^{\alpha }(t,T)$ be the solutions of $G$-BSDEs (\ref{para}) for each
fixed $t\in \lbrack 0,T]$ and $1<\alpha <\beta $. Then $\mu (\cdot ,\cdot
)\in H_{G}^{\alpha }(\Delta \left( 0,T\right) ).$
\end{lemma}

\begin{proof}
Let 
\begin{equation*}
\Pi _{T}^{n}=\{0=t_{0}<\cdot \cdot \cdot <t_{n}=T\}\text{, }n\in \mathbb{N}
\end{equation*}%
be a sequence of partition on $[0,T],$ where $\left\Vert \Pi
_{T}^{n}\right\Vert :=\max \{t_{i+1}-t_{i},0\leq i\leq n-1\}\leq \frac{1}{n}%
. $ Based on the $H_{G}^{\alpha }$-uniform continuity of $\mu (t_{i},\cdot )$
for each fixed $t_{i},$ we derive from Lemma \ref{contin_para} that there
exists a constant $n_{\varepsilon }\in \mathbb{N}$ such that for each $n\geq
n_{\varepsilon },$%
\begin{equation}
\underset{i\leq n}{\sup }\text{ }\underset{\left\vert t-t_{i}\right\vert
\leq \frac{1}{n}}{\sup }\text{ }\mathbb{\hat{E}}\Biggl[\left(
\int_{t}^{T}\left\vert \mu (t,s)-\mu (t_{i},s)\right\vert ^{2}ds\right)
^{\alpha /2}\Biggr]<\frac{\varepsilon }{2},\text{ for each }\varepsilon >0.
\label{Z_space2}
\end{equation}%
For each fixed $n\geq n_{\varepsilon }$ and $i$, since $\mu (t_{i},\cdot
)\in H_{G}^{\alpha }(t_{i},T),$ there exists a sequence $\tilde{\mu}%
_{m}(t_{i},s)=\underset{j=0}{\overset{m-1}{\sum }}\tilde{\mu}%
(t_{i},s_{j}^{i})$ $I_{[s_{j}^{i},s_{j+1}^{i})}(s)\in M_{G}^{0}(t_{i},T),$
where $\Pi _{T}^{m}=\{t_{i}=s_{0}^{i}<\cdot \cdot \cdot <s_{m}^{i}=T\}$ is a
partition of $[t_{i},T]$ and $\tilde{\mu}(t_{i},s_{j}^{i})\in Lip(\Omega
_{s_{j}^{i}})$ for each $0\leq j\leq m-1,$ $m\in 
\mathbb{N}
.$ This implies that there exists a constant $m_{\varepsilon }\in \mathbb{N}$
such that for each $m\geq m_{\varepsilon },$ 
\begin{equation}
\underset{i\leq n}{\sup }\text{ }\mathbb{\hat{E}}\left[ \left(
\int_{t_{i}}^{T}\left\vert \tilde{\mu}_{m}(t_{i},s)-\mu (t_{i},s)\right\vert
^{2}ds\right) ^{\alpha /2}\right] <\frac{\varepsilon }{2},\text{ for each }%
\varepsilon >0.  \label{Z_space1}
\end{equation}%
For each $(t,s)\in \Delta \left[ 0,T\right] $, $n\geq n_{\varepsilon }$ and $%
m\geq m_{\varepsilon },$ we define 
\begin{equation*}
\mu _{nm}(t,s)=\underset{i=0}{\overset{n-1}{\sum }}\underset{j=0}{\overset{%
m-1}{\sum }}\tilde{\mu}(t_{i},s_{j}^{i})I_{\Delta _{ij}\left[ 0,T\right]
}(t,s)\text{ and }\mu _{n}(t,s)=\underset{i=0}{\overset{n-1}{\sum }}\mu
(t_{i},s)I_{[t_{i},t_{i+1})}(t)I_{\Delta \left[ 0,T\right] }\left(
t,s\right) ,
\end{equation*}%
where $\Delta _{ij}\left[ 0,T\right] =\left( [t_{i},t_{i+1})\times \lbrack
s_{j}^{i},s_{j+1}^{i})\right) \cap \Delta \left[ 0,T\right] .$ By adding
more points in partition, it is easy to check that $\mu $$_{nm}(\cdot ,\cdot
)\in M_{G}^{0}(\Delta \left( 0,T\right) ).$

Next, we prove $\mu $$_{nm}$ converges to $\mu $ in $H_{G}^{\alpha }(\Delta
\left( 0,T\right) ).$ Noting that for each $n\geq n_{\varepsilon }$ and $%
m\geq m_{\varepsilon },$%
\begin{equation}
\begin{array}{l}
\displaystyle\mathbb{\hat{E}}\Biggl[\int\nolimits_{0}^{T}\left(
\int_{t}^{T}\left\vert \mu _{nm}(t,s)-\mu (t,s)\right\vert ^{2}ds\right)
^{\alpha /2}dt\Biggr] \\ 
\displaystyle\leq C\left\{ \mathbb{\hat{E}}\Biggl[\int\nolimits_{0}^{T}%
\left( \int_{t}^{T}\left\vert \mu _{nm}(t,s)-\mu _{n}(t,s)\right\vert
^{2}ds\right) ^{\alpha /2}dt\Biggr]+\mathbb{\hat{E}}\Biggl[%
\int\nolimits_{0}^{T}\left( \int_{t}^{T}\left\vert \mu _{n}(t,s)-\mu
(t,s)\right\vert ^{2}ds\right) ^{\alpha /2}dt\Biggr]\right\} \\ 
\displaystyle\leq C\Biggl\{\underset{i=0}{\overset{n-1}{\sum }}\mathbb{\hat{E%
}}\Biggl[\Biggl(\int_{t_{i}}^{T}\Biggl\vert\underset{j=0}{\overset{m-1}{\sum 
}}\tilde{\mu}(t_{i},s_{j}^{i})I_{[s_{j}^{i},s_{j+1}^{i})}(s)-\mu (t_{i},s)%
\Biggr\vert^{2}ds\Biggr)^{\alpha /2}\Biggr]\left( t_{i+1}-t_{i}\right) %
\Biggr. \\ 
\displaystyle\text{ \ \ }+\Biggl.\underset{i\leq n}{\sup }\text{ }\underset{%
\left\vert t-t_{i}\right\vert \leq \frac{1}{n}}{\sup }\text{ }\mathbb{\hat{E}%
}\Biggl[\Biggl(\int_{t}^{T}\left\vert \mu (t_{i},s)-\mu (t,s)\right\vert
^{2}ds\Biggr)^{\alpha /2}\Biggr]\Biggr\},%
\end{array}
\label{Z_space3}
\end{equation}%
where the constant $C$ depends on $\alpha ,T.$ Consequently, by (\ref%
{Z_space2})-(\ref{Z_space3}), we obtain that for each $\varepsilon >0,$
there exists constants $n_{\varepsilon },m_{\varepsilon }$ such that for
each $n\geq n_{\varepsilon }$ and $m\geq m_{\varepsilon }$ satisfying%
\begin{equation*}
\mathbb{\hat{E}}\Biggl[\int\nolimits_{0}^{T}\Biggl(\int_{t}^{T}\left\vert
\mu (t,s)-\mu _{nm}(t,s)\right\vert ^{2}ds\Biggr)^{\alpha /2}dt\Biggr]\leq
C\varepsilon .
\end{equation*}%
Since $\varepsilon $ is arbitrary, we get the conclusion.
\end{proof}

\begin{remark}
\label{eqdengjia}Suppose that $(Y,Z,K)\in \mathfrak{\tilde{S}}_{G}^{\alpha
}(\Delta \left( 0,T\right) )$ is the solution of $G$-BSVIE (\ref{G-BSVIE2})
for any $1<\alpha <\beta $. For $\left( t,r\right) \in \Delta \left[ 0,T%
\right] ,$ we set%
\begin{equation}
\lambda \left( t,r\right)
=Y(t)-\int_{t}^{r}f(t,s,Z(t,s))ds+\int_{t}^{r}Z(t,s)dB_{s}+K(t,r)
\label{remark1}
\end{equation}%
and%
\begin{equation}
\mu \left( t,r\right) =Z\left( t,r\right) ,\text{ }\eta \left( t,r\right)
=K\left( t,r\right) .  \label{remark2}
\end{equation}%
Then by (\ref{remark1}) and (\ref{remark2}), we can easily verify that $%
(\lambda (t,\cdot ),\mu (t,\cdot ),\eta (t,\cdot ))\in \mathfrak{S}%
_{G}^{\alpha }(t,T)$ satisfies $G$-BSDEs (\ref{para}) for each fixed $t\in %
\left[ 0,T\right] .$
\end{remark}

Now, we study the existence, uniqueness\ and the continuity of the solution
to $G$-BSVIE (\ref{noY}).

\begin{theorem}
\label{dengjia}Suppose that $\phi (\cdot )$ is $L_{G}^{\beta }$-continuous
and $f$ satisfies (H1)-(H4) for some $\beta >1$. Then $G$-BSVIE (\ref{noY})
has a unique solution $(Y,Z,K)\in \mathfrak{\tilde{S}}_{G}^{\alpha }(\Delta
\left( 0,T\right) )$ for each $1<\alpha <\beta $. Moreover, $Y(\cdot
),K\left( \cdot ,T\right) $ are $L_{G}^{\alpha }$-continuous and $Z\left(
\cdot ,\cdot \right) $ is $H_{G}^{\alpha }$-continuous$.$
\end{theorem}

\begin{proof}
\textbf{Existence.} Recalling Theorem \ref{Djie}, $G$-BSDEs admit a unique
solution $(\lambda (t,\cdot ),\mu (t,\cdot ),\eta (t,\cdot ))\in \mathfrak{S}%
_{G}^{\alpha }(t,T)$ for each fixed $t\in \lbrack 0,T].$ Set 
\begin{equation}
Y(t)=\lambda (t,t),\text{ }Z(t,s)=\mu (t,s),\text{ }K(t,s)=\eta (t,s).
\label{dengjiasolution}
\end{equation}%
Then, it is easy to check that $(Y,Z,K)$ satisfy $G$-BSVIE (\ref{noY}) and
the properties of $\eta $ also hold for $K$. Moreover, by Lemma \ref%
{contin_para} and Lemma \ref{Z_space}, we have $Z\in \tilde{H}_{G}^{\alpha
}(\Delta \left( 0,T\right) ).$ Then, it suffices to prove $Y\left( \cdot
\right) \in \tilde{M}_{G}^{\alpha }\left( 0,T\right) .$

Note that $\lambda (t,t)=\mathbb{\hat{E}}_{t}\left[ \phi
(t)+\int_{t}^{T}f(t,s,\mu (t,s))ds\right] =Y(t)\in L_{G}^{\alpha }(\Omega
_{t})$ for each $t\in \left[ 0,T\right] .$ Let%
\begin{equation*}
\Pi _{T}^{n}=\{0=t_{0}<\cdot \cdot \cdot <t_{n}=T\},\text{ }n\in \mathbb{N}
\end{equation*}%
be a sequence of partition on $[0,T],$ where $\left\Vert \Pi
_{T}^{n}\right\Vert :=\max \{t_{i+1}-t_{i},0\leq i\leq n-1\}\leq \frac{1}{n}.
$ Define 
\begin{equation*}
Y_{n}(t)=\overset{n-1}{\underset{i=0}{\sum }}Y_{i}(t)I_{[t_{i},t_{i+1})}(t),%
\quad t\in \lbrack 0,T],
\end{equation*}%
where $Y_{i}(t)=\mathbb{\hat{E}}_{t}\left[ \phi
(t_{i})+\int_{t_{i}}^{T}f(t_{i},s,\mu (t_{i},s))ds\right] \in L_{G}^{\alpha
}(\Omega _{t})$ for any $i$ and $t\in \lbrack 0,T].$ Then we have $%
Y_{n}(\cdot )\in M_{G}^{\alpha }(0,T).$ For any $t\in \lbrack t_{i},t_{i+1}),
$ we derive from (H2) that%
\begin{equation}
\begin{array}{ll}
\displaystyle\mathbb{\hat{E}}\left[ \left\vert Y(t)-Y_{i}(t)\right\vert
^{\alpha }\right] \leq  & \displaystyle C\left\{ \mathbb{\hat{E}}\left[
\left\vert \phi (t)-\phi (t_{i})\right\vert ^{\alpha }\right] +\mathbb{\hat{E%
}}\left[ \left( \int_{t}^{T}\left\vert \mu (t,s)-\mu (t_{i},s)\right\vert
^{2}ds\right) ^{\frac{\alpha }{2}}\right] \right.  \\ 
& \displaystyle\left. +\mathbb{\hat{E}}\left[ \left\vert
\int_{t}^{T}f(t,s,\mu (t_{i},s))-f(t_{i},s,\mu (t_{i},s))ds\right\vert
^{\alpha }\right] +\mathbb{\hat{E}}\left[ \left\vert
\int_{t_{i}}^{t}f(t_{i},s,\mu (t_{i},s))ds\right\vert ^{\alpha }\right]
\right\} ,%
\end{array}
\label{1.dengjia}
\end{equation}%
where the constant $C$ depends on $L,\alpha ,T.$ The estimates for the above
term will be showed as follows. Based on the uniform continuity of $\phi $
and $\mu $ (see Lemma \ref{uni_f} and Lemma \ref{contin_para}), there exists
a constant $n_{1}\in \mathbb{N}$ such that for each $n\geq n_{1},$%
\begin{equation}
\underset{i\leq n}{\sup }\text{ }\underset{\left\vert t-t_{i}\right\vert
\leq \frac{1}{n}}{\sup }\left\{ \mathbb{\hat{E}}\left[ \left\vert \phi
(t)-\phi (t_{i})\right\vert ^{\alpha }\right] +\mathbb{\hat{E}}\left[ \left(
\int_{t}^{T}\left\vert \mu (t,s)-\mu (t_{i},s)\right\vert ^{2}ds\right) ^{%
\frac{\alpha }{2}}\right] \right\} <\frac{\varepsilon }{3},\text{ for each }%
\varepsilon >0.  \label{5dengjia}
\end{equation}%
By similar analysis as in (\ref{contin_para_i_2}) and (\ref{contin_para_i_3}%
), we deduce that for each $t\in \lbrack t_{i},t_{i+1})$ and $N\in \mathbb{N}%
,$%
\begin{equation}
\begin{array}{l}
\displaystyle\underset{i\leq n}{\sup }\text{ }\mathbb{\hat{E}}\left[
\left\vert \int_{t}^{T}f(t,s,\mu (t_{i},s))-f(t_{i},s,\mu
(t_{i},s))ds\right\vert ^{\alpha }\right]  \\ 
\displaystyle\leq C_{1}\left\{ \underset{i\leq n}{\sup }\text{ }\mathbb{\hat{%
E}}\left[ \left\vert \int_{t_{i}}^{T}\left\vert \mu (t_{i},s))\right\vert
I_{\left\{ \left\vert \mu (t_{i},s)\right\vert >N\right\} }ds\right\vert
^{\alpha }\right] +\underset{i\leq n}{\sup }\text{ }\mathbb{\hat{E}}\left[
\left\vert \int_{t}^{T}\left\vert f(t,s,0)-f(t_{i},s,0)\right\vert
ds\right\vert ^{\alpha }\right] \right.  \\ 
\displaystyle\text{ \ \ }\left. +\underset{i\leq n}{\sup }\text{ }\mathbb{%
\hat{E}}\left[ \left\vert \int_{t}^{T}f(t,s,\mu (t_{i},s))-f(t_{i},s,\mu
(t_{i},s))I_{\left\{ \left\vert \mu (t_{i},s)\right\vert \leq N\right\}
}ds\right\vert ^{\alpha }\right] \right\} 
\end{array}
\label{3.dengjia}
\end{equation}%
and%
\begin{equation*}
\begin{array}{ll}
\displaystyle\underset{i\leq n}{\sup }\text{ }\mathbb{\hat{E}}\left[
\left\vert \int_{t_{i}}^{T}\left\vert \mu (t_{i},s))\right\vert I_{\left\{
\left\vert \mu (t_{i},s)\right\vert >N\right\} }ds\right\vert ^{\alpha }%
\right]  & \displaystyle\leq \frac{C_{2}}{N^{\alpha \gamma }}\underset{i\leq
n}{\sup }\text{ }\mathbb{\hat{E}}\left[ \left\vert
\int_{t_{i}}^{T}\left\vert \mu (t_{i},s))\right\vert ^{2}ds\right\vert ^{%
\frac{\left( 1+\gamma \right) \alpha }{2}}\right]  \\ 
& \displaystyle\leq C_{2}N^{-\frac{\alpha \left( \beta -\alpha \right) }{%
\beta +\alpha }},%
\end{array}%
\end{equation*}%
where $0<\gamma =\frac{\beta -\alpha }{\beta +\alpha }<1$ and $\alpha
<\left( 1+\gamma \right) \alpha =\frac{2\alpha \beta }{\beta +\alpha }<\beta 
$. The constants $C_{1},C_{2}$ are given by $C_{1}=C(\alpha ,L,T)$ and $%
C_{2}=C(\alpha ,\beta ,T,\underline{\sigma }).$ Then it holds that for each $%
\varepsilon >0,$ there exists a constant $N_{0}\in \mathbb{N}$ such that 
\begin{equation*}
\underset{i\leq n}{\sup }\text{ }\mathbb{\hat{E}}\left[ \left\vert
\int_{t_{i}}^{T}\left\vert \mu (t_{i},s))\right\vert I_{\left\{ \left\vert
\mu (t_{i},s)\right\vert >N_{0}\right\} }ds\right\vert ^{\alpha }\right] <%
\frac{\varepsilon }{6},\text{ for each }\varepsilon >0.
\end{equation*}%
Taking $N=N_{0}$ in (\ref{3.dengjia}) and recalling Lemma \ref{uni_f}, we
derive from (\ref{3.dengjia}) that there exists a constant $n_{2}\in \mathbb{%
N}$ such that for each $n\geq n_{2},$%
\begin{equation}
\underset{i\leq n}{\sup }\text{ }\underset{\left\vert t-t_{i}\right\vert
\leq \frac{1}{n}}{\sup }\mathbb{\hat{E}}\left[ \left\vert
\int_{t}^{T}f(t,s,\mu (t_{i},s))-f(t_{i},s,\mu (t_{i},s))ds\right\vert
^{\alpha }\right] <\frac{\varepsilon }{3},\text{ for each }\varepsilon >0.
\label{4dengjia}
\end{equation}%
Making use of the estimates in (\ref{contin_para_4}), we get%
\begin{equation*}
\begin{array}{l}
\displaystyle\underset{i\leq n}{\sup }\text{ }\underset{\left\vert
t-t_{i}\right\vert \leq \frac{1}{n}}{\sup }\text{ }\mathbb{\hat{E}}\left[
\left\vert \int_{t_{i}}^{t}f(t_{i},s,\mu (t_{i},s))ds\right\vert ^{\alpha }%
\right]  \\ 
\displaystyle\leq (t-t_{i})^{\alpha -1}\underset{i\leq n}{\sup }\text{ }%
\mathbb{\hat{E}}\left[ \int\nolimits_{t_{i}}^{T}\left\vert
f(t_{i},s,0)\right\vert ^{\alpha }ds\right] +L^{\alpha }(t-t_{i})^{\frac{%
\alpha }{2}}\underset{i\leq n}{\sup }\text{ }\mathbb{\hat{E}}\left[ \left(
\int_{t_{i}}^{T}\left\vert \mu (t_{i},s)\right\vert ^{2}ds\right) ^{\alpha
/2}\right]  \\ 
\displaystyle\leq C\left\Vert \Pi _{T}^{n}\right\Vert ^{\left( \alpha
-1\right) \wedge \frac{\alpha }{2}}\leq Cn^{-\left( \alpha -1\right) \wedge 
\frac{\alpha }{2}},%
\end{array}%
\end{equation*}%
where the constant $C$ depends on $\alpha ,\beta ,T,L,\underline{\sigma }.$
It indicates that there exists a constant $n_{3}\in \mathbb{N}$ such that
for each $n\geq n_{3},$ 
\begin{equation}
\underset{i\leq n}{\sup }\text{ }\underset{\left\vert t-t_{i}\right\vert
\leq \frac{1}{n}}{\sup }\text{ }\mathbb{\hat{E}}\left[ \left\vert
\int_{t_{i}}^{t}f(t_{i},s,\mu (t_{i},s))ds\right\vert ^{\alpha }\right] <%
\frac{\varepsilon }{3},\text{ for each }\varepsilon >0.  \label{6dengjia}
\end{equation}

Therefore, combining (\ref{1.dengjia}) with (\ref{5dengjia})-(\ref{6dengjia}%
), we obtain that there exists a constant $n_{0}=\max \{n_{1},n_{2},n_{3}\}$%
\newline
$\in \mathbb{N}$ such that for each $n\geq n_{0},$%
\begin{equation*}
\mathbb{\hat{E}}\left[ \int_{0}^{T}\left\vert Y(t)-Y_{n}(t)\right\vert
^{\alpha }dt\right] \leq T\text{ }\underset{i\leq n}{\sup }\text{ }\underset{%
\left\vert t-t_{i}\right\vert \leq \frac{1}{n}}{\sup }\text{ }\mathbb{\hat{E}%
}\left[ \left\vert Y(t)-Y_{i}(t)\right\vert ^{\alpha }\right] \leq
C\varepsilon ,
\end{equation*}%
\newline
where $C$ depends on $\alpha ,\beta ,T,L,\underline{\sigma }.$ Since $%
\varepsilon >0$ is arbitrary, we conclude $\underset{n\rightarrow \infty }{%
\lim }\mathbb{\hat{E}}\left[ \int_{0}^{T}\left\vert Y(t)-Y_{n}(t)\right\vert
^{\alpha }dt\right] =0$, which leads to $Y(\cdot )\in M_{G}^{\alpha }(0,T).$
Then $(Y,Z,K)$ in (\ref{dengjiasolution}) is the solution of $G$-BSVIE (\ref%
{noY}).\newline
\textbf{Uniqueness}. Let $(Y_{1},Z_{1},K_{1})$ and $(Y_{2},Z_{2},K_{2})$ be
two solutions of $G$-BSVIE (\ref{noY}). Recalling Remark \ref{eqdengjia}, we
have for each $\left( t,r\right) \in \Delta \left[ 0,T\right] $,%
\begin{equation*}
\lambda _{i}\left( t,r\right)
=Y_{i}(t)-\int_{t}^{r}f(t,s,Z_{i}(t,s))ds+%
\int_{t}^{r}Z_{i}(t,s)dB_{s}+K_{i}(t,r),\text{ }i=1,2
\end{equation*}%
and $\mu _{i}\left( t,r\right) =Z_{i}\left( t,r\right) ,$ $\eta _{i}\left(
t,r\right) =K_{i}\left( t,r\right) ,$ where $\left( \lambda _{i},\mu
_{i},\eta _{i}\right) $ is the solution of $G$-BSDEs (\ref{para}). Based on
the uniqueness of the solution to $G$-BSDEs (\ref{para}), we deduce $\left(
\lambda _{1},\mu _{1},\eta _{1}\right) =\left( \lambda _{2},\mu _{2},\eta
_{2}\right) .$ Set$\ r=t,$ it holds that for each $t\in \lbrack 0,T],$%
\begin{equation*}
\lambda _{1}\left( t,t\right) =Y_{1}(t)=Y_{2}(t)=\lambda _{2}\left(
t,t\right) .
\end{equation*}%
Moreover, we have $Z_{1}=Z_{2},$ $K_{1}=K_{2}.$ Then the uniqueness
arguement follows from it.\newline
\textbf{Continuity. }The continuity of $Z(\cdot ,\cdot )$ can be directly
obtained from $\mu (\cdot ,\cdot )$. For the term $Y$, it is easy to check
that for each $t^{\prime },t\in \left[ 0,T\right] $ with $t^{\prime }>t,$%
\begin{equation}
\begin{array}{ll}
\displaystyle\mathbb{\hat{E}}\left[ \left\vert Y(t)-Y(t^{\prime
})\right\vert ^{\alpha }\right]  & \displaystyle=\mathbb{\hat{E}}\left[
\left\vert \lambda \left( t,t\right) -\lambda \left( t^{\prime },t^{\prime
}\right) \right\vert ^{\alpha }\right]  \\ 
& \displaystyle\leq C\left( \alpha \right) \left\{ \mathbb{\hat{E}}\left[
\left\vert \lambda \left( t,t\right) -\lambda \left( t,t^{\prime }\right)
\right\vert ^{\alpha }\right] +\mathbb{\hat{E}}\left[ \underset{u\in \lbrack
t\vee t^{\prime },T]}{\sup }\left\vert \lambda \left( t,u\right) -\lambda
\left( t^{\prime },u\right) \right\vert ^{\alpha }\right] \right\} .%
\end{array}
\label{noYcontin}
\end{equation}%
Then by Lemma \ref{contin_para}, we derive $Y(\cdot )$ is $L_{G}^{\alpha }$%
-continuous$.$ For the term $K$, we have for each $t^{\prime },t\in \left[
0,T\right] $ with $t^{\prime }>t,$%
\begin{equation}
\begin{array}{ll}
\displaystyle\mathbb{\hat{E}}\left[ \left\vert K(t,T)-K(t^{\prime
},T)\right\vert ^{\alpha }\right] \leq  & \displaystyle C\left\{ \mathbb{%
\hat{E}}\left[ \left\vert Y(t)-Y(t^{\prime })\right\vert ^{\alpha }\right] +%
\mathbb{\hat{E}}\left[ \left\vert \phi (t)-\phi (t^{\prime })\right\vert
^{\alpha }\right] +\mathbb{\hat{E}}\Biggl[\left\vert \int_{t}^{t^{\prime
}}\left\vert f(t,s,Z(t,s))\right\vert ds\right\vert ^{\alpha }\Biggr]\right. 
\\ 
& \displaystyle\left. +\mathbb{\hat{E}}\Biggl[\left\vert \int_{t^{\prime
}}^{T}\left\vert f(t,s,Z(t,s))-f(t^{\prime },s,Z(t^{\prime },s))\right\vert
ds\right\vert ^{\alpha }\Biggr]+\mathbb{\hat{E}}\Biggl[\left\vert
\int_{t}^{t^{\prime }}Z(t,s)dB_{s}\right\vert ^{\alpha }\Biggr]\right.  \\ 
& \displaystyle\left. +\mathbb{\hat{E}}\Biggl[\Biggl\vert\int_{t^{\prime
}}^{T}\left( Z(t,s)-Z(t^{\prime },s)\right) dB_{s}\Biggr\vert^{\alpha }%
\Biggr]\right\} ,%
\end{array}
\label{noYK_contin}
\end{equation}%
where the constant $C$ depends on $\alpha .$ Following the proof of Lemma %
\ref{contin_para}, we deduce from the continuity of $Y\left( \cdot \right) $
and $Z\left( \cdot ,\cdot \right) $ that $K\left( \cdot ,T\right) $ is $%
L_{G}^{\alpha }$-continuous.
\end{proof}

In the following, we establish the estimates of $G$-BSVIE (\ref{noY}).

\begin{proposition}
\label{noYguji1}Assume that $\phi (\cdot )$ is $L_{G}^{\beta }$-continuous
and $f$ satisfies (H1)-(H4) for some $\beta >1$. Let $(Y,Z,K)\in \mathfrak{%
\tilde{S}}_{G}^{\alpha }(\Delta \left( 0,T\right) )$ be the solution of $G$%
-BSVIE (\ref{noY}) for any $1<\alpha <\beta $. Then, for each $t\in \lbrack
0,T]$ and $\delta >0$ with $\alpha +\delta <\beta ,$%
\begin{equation}
\left\vert Y(t)\right\vert ^{\alpha }\leq C\mathbb{\hat{E}}_{t}\left[
\left\vert \phi (t)\right\vert ^{\alpha }+\left( \int_{t}^{T}\left\vert
f(t,s,0)\right\vert ds\right) ^{\alpha }\right] ,  \label{noYguji_1}
\end{equation}%
\begin{eqnarray}
\mathbb{\hat{E}}\left[ \left( \int_{t}^{T}\left\vert Z(t,s)\right\vert
^{2}ds\right) ^{\alpha /2}\right]  &\leq &C_{\delta }\left\{ \mathbb{\hat{E}}%
\left[ \left\vert \phi (t)\right\vert ^{\alpha +\delta }\right] ^{\frac{%
\alpha }{\alpha +\delta }}+\mathbb{\hat{E}}\left[ \left(
\int_{t}^{T}\left\vert f(t,s,0)\right\vert ds\right) ^{\alpha +\delta }%
\right] ^{\frac{\alpha }{\alpha +\delta }}\right.   \notag \\
&&\left. +\mathbb{\hat{E}}\left[ \left\vert \phi (t)\right\vert ^{\alpha
+\delta }\right] ^{\frac{\alpha }{2\left( \alpha +\delta \right) }}\times 
\mathbb{\hat{E}}\left[ \left( \int_{t}^{T}\left\vert f(t,s,0)\right\vert
ds\right) ^{\alpha +\delta }\right] ^{\frac{\alpha }{2\left( \alpha +\delta
\right) }}\right\} ,  \label{noYguji_2}
\end{eqnarray}%
\begin{equation}
\mathbb{\hat{E}}\left[ \left\vert K(t,T)\right\vert ^{\alpha }\right] \leq
C_{\delta }\left\{ \mathbb{\hat{E}}\left[ \left\vert \phi (t)\right\vert
^{\alpha +\delta }\right] ^{\frac{\alpha }{\alpha +\delta }}+\mathbb{\hat{E}}%
\left[ \left( \int_{t}^{T}\left\vert f(t,s,0)\right\vert ds\right) ^{\alpha
+\delta }\right] ^{\frac{\alpha }{\alpha +\delta }}\right\} ,
\label{noYguji_3}
\end{equation}%
where the constants $C,C_{\delta }$ depend on $\alpha ,T,\underline{\sigma }%
,L$ and $\alpha ,\delta ,T,\underline{\sigma },L$ respectively$.$
\end{proposition}

\begin{proof}
From Theorem \ref{dengjia}, it holds that $\lambda \left( t,t\right)
=Y\left( t\right) ,$ $Z\left( t,r\right) =\mu \left( t,r\right) ,$ $K\left(
t,r\right) =\eta \left( t,r\right) $ for each $\left( t,r\right) \in \Delta %
\left[ 0,T\right] .$\ Then by Propositon \ref{Dguji1}, we derive%
\begin{equation}
\left\vert \lambda (t,r)\right\vert ^{\alpha }\leq C\mathbb{\hat{E}}_{r}%
\left[ \left\vert \phi (t)\right\vert ^{\alpha }+\left(
\int_{r}^{T}\left\vert f(t,s,0)\right\vert ds\right) ^{\alpha }\right] 
\label{noYguji_4}
\end{equation}%
and
\begin{equation}
\begin{array}{ll}
\displaystyle\mathbb{\hat{E}}\left[ \left( \int_{t}^{T}\left\vert
Z(t,s)\right\vert ^{2}ds\right) ^{\alpha /2}\right] \leq & \displaystyle %
C\left\{ \mathbb{\hat{E}}\left[ \underset{u\in \lbrack t,T]}{\sup }%
\left\vert \lambda (t,u)\right\vert ^{\alpha }\right] +\mathbb{\hat{E}}\left[
\underset{u\in \lbrack t,T]}{\sup }\left\vert \lambda (t,u)\right\vert
^{\alpha }\right] ^{1/2}\right. \\ 
& \displaystyle\left. \times \mathbb{\hat{E}}\left[ \left(
\int_{t}^{T}\left\vert f(t,s,0)\right\vert ds\right) ^{\alpha }\right]
^{1/2}\right\},%
\end{array}
\label{noYguji_5}
\end{equation}%
where the constant $C$ depends on $\alpha ,T,\underline{\sigma },L$.

Taking $r=t$ in (\ref{noYguji_4}), the estimate (\ref{noYguji_1}) is
obtained directly. From Theorem \ref{Esup}, we have for each $\delta >0$
with $\alpha +\delta <\beta ,$%
\begin{equation}
\begin{array}{ll}
\displaystyle\mathbb{\hat{E}}\left[ \underset{u\in \lbrack t,T]}{\sup }%
\left\vert \lambda (t,u)\right\vert ^{\alpha }\right] & \displaystyle\leq
C_{\delta }\mathbb{\hat{E}}\left[ \underset{u\in \lbrack t,T]}{\sup }\mathbb{%
\hat{E}}_{u}\left[ \left\vert \phi (t)\right\vert ^{\alpha }+\left(
\int_{t}^{T}\left\vert f(t,s,0)\right\vert ds\right) ^{\alpha }\right] %
\right] \\ 
& \displaystyle\leq C_{\delta }\Biggl\{\mathbb{\hat{E}}\left[ \left\vert
\phi (t)\right\vert ^{\alpha +\delta }\right] ^{\frac{\alpha }{\alpha
+\delta }}+\mathbb{\hat{E}}\Biggl[\left( \int_{t}^{T}\left\vert
f(t,s,0)\right\vert ds\right) ^{\alpha +\delta }\Biggr]^{\frac{\alpha }{%
\alpha +\delta }}\Biggr\},%
\end{array}
\label{noYguji_7}
\end{equation}%
where the constant $C_{\delta }$ depends on $\alpha ,T,\underline{\sigma }%
,L,\delta .$ Consequently,\ (\ref{noYguji_2}) follows from (\ref{noYguji_5})
and (\ref{noYguji_7}). Similarly, we can get the estimate for $K.$
\end{proof}

\begin{proposition}
\label{noYguji2}Assume that $\phi _{i}(\cdot )$ is $L_{G}^{\beta }$%
-continuous and $f_{i}$ satisfies (H1)-(H4) for some $\beta >1$ and $i=1,2$.
Let $(Y_{i},Z_{i},K_{i})\in \mathfrak{\tilde{S}}_{G}^{\alpha }(\Delta \left(
0,T\right) )$ be the solution of $G$-BSVIE of the form:%
\begin{equation*}
Y^{i}(t)=\phi
^{i}(t)+\int_{t}^{T}f^{i}(t,s,Z^{i}(t,s))ds-%
\int_{t}^{T}Z^{i}(t,s)dB_{s}-K^{i}(t,T),\quad t\in \lbrack 0,T],
\end{equation*}%
where $1<\alpha <\beta .$ Set $\hat{Y}=Y^{1}-Y^{2},$ $\hat{Z}=Z^{1}-Z^{2},\
K=K^{1}-K^{2}\ $and $\hat{\phi}=\phi ^{1}-\phi ^{2}.$ Then there exist
positive constants $C:=C(\alpha ,T,\underline{\sigma },L)$ and $C_{\delta
}:=(\alpha ,T,\underline{\sigma },L,\delta )$ such that for each $\delta >0$
with $\alpha +\delta <\beta ,$%
\begin{equation*}
|\hat{Y}(t)|^{\alpha }\leq C\mathbb{\hat{E}}_{t}\left[ |\hat{\phi}%
(t)|^{\alpha }+\left( \int_{t}^{T}|\hat{f}(t,s)|ds\right) ^{\alpha }\right] ,
\end{equation*}%
\begin{equation}
\begin{array}{ll}
\displaystyle\mathbb{\hat{E}}\Biggl[\Biggl(\int_{t}^{T}|\hat{Z}(t,s)|^{2}ds%
\Biggr)^{\alpha /2}\Biggr]\leq  & \displaystyle C_{\delta }\Biggl\{\mathbb{%
\hat{E}}\left[ |\hat{\phi}(t)|^{\alpha +\delta }\right] ^{\frac{\alpha }{%
\alpha +\delta }}+\mathbb{\hat{E}}\Biggl[\Biggl(\int_{t}^{T}|\hat{f}(t,s)|ds%
\Biggr)^{\alpha +\delta }\Biggr]^{\frac{\alpha }{\alpha +\delta }}\Biggr. \\ 
& \displaystyle+\Biggl.\Biggl(\mathbb{\hat{E}}\left[ |\hat{\phi}(t)|^{\alpha
+\delta }\right] ^{\frac{\alpha }{2\left( \alpha +\delta \right) }}+\mathbb{%
\hat{E}}\Biggl[\Biggl(\int_{t}^{T}|\hat{f}(t,s)|ds\Biggr)^{\alpha +\delta }%
\Biggr]^{\frac{\alpha }{2\left( \alpha +\delta \right) }}\Biggr)\times 
\underset{i=1}{\overset{2}{\sum }}M_{i}\Biggr\},%
\end{array}%
\end{equation}%
where $\hat{f}(t,s)=f^{1}(t,s,Z^{2}(t,s))-f^{2}(t,s,Z^{2}(t,s))$ and 
\begin{equation*}
M_{i}=\mathbb{\hat{E}}\left[ \left\vert \phi ^{i}(t)\right\vert ^{\alpha
+\delta }\right] ^{\frac{\alpha }{2\left( \alpha +\delta \right) }}+\mathbb{%
\hat{E}}\Biggl[\left( \int_{t}^{T}\left\vert f^{i}(t,s,0)\right\vert
ds\right) ^{\alpha +\delta }\Biggr]^{\frac{\alpha }{2\left( \alpha +\delta
\right) }}\text{ for }i=1,2.
\end{equation*}
\end{proposition}

\begin{proof}
The proof is the similar to the analysis in Proposition \ref{noYguji1}.
\end{proof}

\subsection{The $G$-BSVIEs with generators depending on $Y$}

\noindent

We now present a priori estimate for $G$-BSVIE (\ref{G-BSVIE2}).

\begin{proposition}
\label{Yguji1}Assume that $\phi (\cdot )$ is $L_{G}^{\beta }$-continuous and 
$f$ satisfies (H1)-(H4) for some $\beta >1$. Let $(Y,Z,K)\in \mathfrak{%
\tilde{S}}_{G}^{\alpha }(\Delta \left( 0,T\right) )$ be the unique adapted
solution of $G$-BSVIE (\ref{G-BSVIE2}) for any $1<\alpha <\beta .$ Then
there exist positive constants $C:=C\left( \alpha ,T,\underline{\sigma }%
,L\right) $ and $C_{\delta }:=C_{\delta }\left( \alpha ,\delta ,T,\underline{%
\sigma },L\right) $ such that for any $t\in \lbrack 0,T]$ and $\delta >0$
with $\alpha +\delta <\beta ,$ 
\begin{equation}
\mathbb{\hat{E}}\left[ \int_{0}^{T}\left\vert Y(t)\right\vert ^{\alpha }dt%
\right] \leq C\int_{0}^{T}\mathbb{\hat{E}}\left[ \left\vert \phi
(t)\right\vert ^{\alpha }+\left( \int_{t}^{T}\left\vert
f_{0}(t,s)\right\vert ds\right) ^{\alpha }\right] dt,  \label{Yguji_6}
\end{equation}%
\begin{equation}
\begin{array}{ll}
\displaystyle\mathbb{\hat{E}}\left[ \left( \int_{t}^{T}\left\vert
Z(t,s)\right\vert ^{2}ds\right) ^{\alpha /2}\right] \leq  & \displaystyle %
C_{\delta }\left\{ \underset{u\in \lbrack t,T]}{\sup }\left\Vert \phi \left(
u\right) \right\Vert _{L_{G}^{\alpha +\delta }}^{\alpha }+\underset{u\in
\lbrack t,T]}{\sup }\left\Vert \int_{u}^{T}\left\vert f_{0}\left( u,s\right)
\right\vert ds\right\Vert _{L_{G}^{\alpha +\delta }}^{\alpha }\right.  \\ 
& \displaystyle+\Biggl.\underset{u\in \lbrack t,T]}{\sup }\left\Vert \phi
\left( u\right) \right\Vert _{L_{G}^{\alpha +\delta }}^{^{\frac{\alpha }{2}%
}}\times \underset{u\in \lbrack t,T]}{\sup }\left\Vert
\int_{u}^{T}\left\vert f_{0}\left( u,s\right) \right\vert ds\right\Vert
_{L_{G}^{\alpha +\delta }}^{\frac{\alpha }{2}}\Biggr\},%
\end{array}
\label{Yguji_8}
\end{equation}%
\begin{equation}
\mathbb{\hat{E}}\left[ \left\vert K(t,T)\right\vert ^{\alpha }\right] \leq
C_{\delta }\left\{ \underset{u\in \lbrack t,T]}{\sup }\left\Vert \phi \left(
u\right) \right\Vert _{L_{G}^{\alpha +\delta }}^{\alpha }+\underset{u\in
\lbrack t,T]}{\sup }\left\Vert \int_{u}^{T}\left\vert f_{0}\left( u,s\right)
\right\vert ds\right\Vert _{L_{G}^{\alpha +\delta }}^{\alpha }\right\} ,
\end{equation}%
where $f_{0}\left( t,s\right) =f\left( t,s,0,0\right) .$
\end{proposition}

\begin{proof}
From the estimates in Proposition \ref{noYguji1}, we derive that for each
fixed $t\in \left[ 0,T\right] $ and each $s\geq t,$%
\begin{equation}
\begin{array}{ll}
\displaystyle\mathbb{\hat{E}}_{t}\left[ \left\vert Y(s)\right\vert ^{\alpha }%
\right] & \displaystyle\leq C\mathbb{\hat{E}}_{t}\left[ \mathbb{\hat{E}}_{s}%
\left[ \left\vert \phi (s)\right\vert ^{\alpha }+\left(
\int_{s}^{T}\left\vert f(s,r,Y\left( r\right) ,0)\right\vert dr\right)
^{\alpha }\right] \right] \\ 
& \displaystyle\leq C\mathbb{\hat{E}}_{t}\left[ \left\vert \phi
(s)\right\vert ^{\alpha }+\left( \int_{s}^{T}\left\vert
f(s,r,0,0)\right\vert dr\right) ^{\alpha }\right] +CL^{\alpha }T^{\alpha
-1}\int_{s}^{T}\mathbb{\hat{E}}_{t}\left[ \left\vert Y\left( r\right)
\right\vert ^{\alpha }\right] dr,%
\end{array}
\label{Yguji1_1}
\end{equation}%
where the constant $C$ depends on $T,\underline{\sigma },\alpha ,L.$
Applying Gronwall's inequality to (\ref{Yguji1_1}), we get for each fixed $%
t\in \left[ 0,T\right] $ and each $s\geq t,$%
\begin{equation*}
\begin{array}{ll}
\displaystyle\mathbb{\hat{E}}_{t}\left[ \left\vert Y(s)\right\vert ^{\alpha }%
\right] \leq & \displaystyle C\left\{ \mathbb{\hat{E}}_{t}\left[ \left\vert
\phi (s)\right\vert ^{\alpha }+\left( \int_{s}^{T}\left\vert
f(s,r,0,0)\right\vert dr\right) ^{\alpha }\right] \right. \\ 
& \displaystyle\left. +\int_{s}^{T}\mathbb{\hat{E}}_{t}\left[ \left\vert
\phi (u)\right\vert ^{\alpha }+\left( \int_{u}^{T}\left\vert
f(u,r,0,0)\right\vert dr\right) ^{\alpha }\right] e^{C\left( u-s\right)
}du\right\} ,%
\end{array}%
\end{equation*}%
where $C=C\left( T,\underline{\sigma },\alpha ,L\right) $ is a constant
varying from line to line. Let $s=t,$ it holds that for any $1<\alpha <\beta
,$%
\begin{equation}
\begin{array}{ll}
\displaystyle\left\vert Y(t)\right\vert ^{\alpha }=\mathbb{\hat{E}}_{t}\left[
\left\vert Y(t)\right\vert ^{\alpha }\right] \leq & \displaystyle C\left\{ 
\mathbb{\hat{E}}_{t}\left[ \left\vert \phi (t)\right\vert ^{\alpha }+\left(
\int_{t}^{T}\left\vert f(t,r,0,0)\right\vert dr\right) ^{\alpha }\right]
\right. \\ 
& \displaystyle\left. +\int_{t}^{T}\mathbb{\hat{E}}_{t}\left[ \left\vert
\phi (u)\right\vert ^{\alpha }+\left( \int_{u}^{T}\left\vert
f(u,r,0,0)\right\vert dr\right) ^{\alpha }\right] e^{C\left( u-t\right)
}du\right\} ,%
\end{array}
\label{Yguji_9}
\end{equation}%
which indicates that 
\begin{equation*}
\begin{array}{ll}
\displaystyle\mathbb{\hat{E}}\left[ \int_{0}^{T}\left\vert Y(t)\right\vert
^{\alpha }dt\right] \leq & \displaystyle C\left\{ \int_{0}^{T}\mathbb{\hat{E}%
}\left[ \left\vert \phi (t)\right\vert ^{\alpha }+\left(
\int_{t}^{T}\left\vert f(t,r,0,0)\right\vert dr\right) ^{\alpha }\right]
dt\right. \\ 
& \displaystyle\left. +e^{CT}\int_{0}^{T}\int_{0}^{T}\mathbb{\hat{E}}\left[
\left\vert \phi (u)\right\vert ^{\alpha }+\left( \int_{u}^{T}\left\vert
f(u,r,0,0)\right\vert dr\right) ^{\alpha }\right] dudt\right\} .%
\end{array}%
\end{equation*}
The constant $C$ depends on $\alpha ,T,\underline{\sigma },L$. Then, (\ref%
{Yguji_6}) follows from it.

For the term $Z,$ it follows from Proposition \ref{noYguji1} that%
\begin{equation}
\begin{array}{l}
\displaystyle\mathbb{\hat{E}}\left[ \left( \int_{t}^{T}\left\vert
Z(t,s)\right\vert ^{2}ds\right) ^{\alpha /2}\right]  \\ 
\displaystyle\leq C_{\delta }\left\{ \left\Vert \phi \left( t\right)
\right\Vert _{L_{G}^{\alpha +\delta }}^{\alpha }+\left\Vert
\int_{t}^{T}\left\vert f(t,s,0,0)\right\vert ds\right\Vert _{L_{G}^{\alpha
+\delta }}^{\alpha }+\mathbb{\hat{E}}\left[ \int_{t}^{T}\left\vert
Y(s)\right\vert ^{\alpha +\delta }ds\right] ^{\frac{\alpha }{\alpha +\delta }%
}+\left\Vert \phi \left( t\right) \right\Vert _{L_{G}^{\alpha +\delta }}^{%
\frac{\alpha }{2}}\right.  \\ 
\text{ \ }\displaystyle\left. \times \left( \left\Vert
\int_{t}^{T}\left\vert f(t,s,0,0)\right\vert ds\right\Vert _{L_{G}^{\alpha
+\delta }}^{\frac{\alpha }{2}}+\mathbb{\hat{E}}\left[ \int_{t}^{T}\left\vert
Y(s)\right\vert ^{\alpha +\delta }ds\right] ^{\frac{\alpha }{2\left( \alpha
+\delta \right) }}\right) \right\} ,%
\end{array}
\label{Yguji_10}
\end{equation}%
where the constant $C_{\delta }$ depends on $\alpha ,T,\underline{\sigma }%
,L,\delta .$ In addition, note that (\ref{Yguji_9}) implies that 
\begin{equation}
\mathbb{\hat{E}}\left[ \int_{t}^{T}\left\vert Y(s)\right\vert ^{\alpha
+\delta }ds\right] ^{\frac{\alpha }{\alpha +\delta }}\leq C\left(
\int_{t}^{T}\mathbb{\hat{E}}\left[ \left\vert \phi (s)\right\vert ^{\alpha
+\delta }+\left( \int_{s}^{T}\left\vert f_{0}(s,r)\right\vert dr\right)
^{\alpha +\delta }\right] ds\right) ^{\frac{\alpha }{\alpha +\delta }}.
\label{Yguji_12}
\end{equation}%
Then, the estimate (\ref{Yguji_8}) follows from (\ref{Yguji_10}) and (\ref%
{Yguji_12}). Using the same method, we obtain the estimates for $K.$
\end{proof}

\begin{proposition}
\label{Yguji2}Suppose that $\phi ^{i}(\cdot )$ is $L_{G}^{\beta }$%
-continuous and $f^{i}$ satisfies (H1)-(H4) for some $\beta >1$ and $i=1,2$.
Let $(Y^{i},Z^{i},K^{i})\in \mathfrak{\tilde{S}}_{G}^{\alpha }(\Delta \left(
0,T\right) )$ be the solution of $G$-BSVIE of the form: 
\begin{equation}
Y^{i}(t)=\phi
^{i}(t)+\int_{t}^{T}f^{i}(t,s,Y^{i}(s),Z^{i}(t,s))ds-%
\int_{t}^{T}Z^{i}(t,s)dB_{s}-K^{i}(t,T),\quad t\in \lbrack 0,T],
\label{Ycha}
\end{equation}%
where $1<\alpha <\beta .$ Set $\hat{Y}=Y^{1}-Y^{2},$ $\hat{Z}=Z^{1}-Z^{2},\
K=K^{1}-K^{2}\ $and $\hat{\phi}=\phi ^{1}-\phi ^{2}.$ Then there exist
positive constants $C:=C\left( \alpha ,T,\underline{\sigma },L\right) $ and $%
C_{\delta }=C_{\delta }\left( \alpha ,\delta ,T,\underline{\sigma },L\right) 
$ such that for any $t\in \lbrack 0,T]$ and $\delta >0$ with $\alpha +\delta
<\beta ,$ 
\begin{equation}
\mathbb{\hat{E}}\left[ \int_{0}^{T}|\hat{Y}(t)|^{\alpha }dt\right] \leq
C\int_{0}^{T}\mathbb{\hat{E}}\left[ |\hat{\phi}\left( t\right) |^{\alpha
}+\left( \int_{t}^{T}|\hat{f}(t,s)|ds\right) ^{\alpha }\right] dt,
\label{Yguji2_1}
\end{equation}%
\begin{equation}
\begin{array}{ll}
\displaystyle\mathbb{\hat{E}}\left[ \left( \int_{t}^{T}\left\vert \hat{Z}%
(t,s)\right\vert ^{2}ds\right) ^{\alpha /2}\right] \leq  & \displaystyle %
C_{\delta }\left\{ \underset{u\in \lbrack t,T]}{\sup }||\hat{\phi}\left(
u\right) ||_{L_{G}^{\alpha +\delta }}^{\alpha }+\underset{u\in \lbrack u,T]}{%
\sup }\left\Vert \int_{u}^{T}|\hat{f}(u,s)|ds\right\Vert _{L_{G}^{\alpha
+\delta }}^{\alpha }\right.  \\ 
& \displaystyle\left. +\left( \underset{u\in \lbrack t,T]}{\sup }||\hat{\phi}%
\left( u\right) ||_{L_{G}^{\alpha +\delta }}^{\frac{\alpha }{2}}+\underset{%
u\in \lbrack t,T]}{\sup }\left\Vert \int_{u}^{T}|\hat{f}(u,s)|ds\right\Vert
_{L_{G}^{\alpha +\delta }}^{\frac{\alpha }{2}}\right) \times \underset{i=1}{%
\overset{2}{\sum }}M_{i}\right\} ,%
\end{array}
\label{Yguji2_2}
\end{equation}%
where $\hat{f}(t,s)=f^{1}(t,s,Y^{2}\left( s\right)
,Z^{2}(t,s))-f^{2}(t,s,Y^{2}\left( s\right) ,Z^{2}(t,s))$ and%
\begin{equation*}
M_{i}=\underset{u\in \lbrack t,T]}{\sup }||\phi ^{i}\left( u\right)
||_{L_{G}^{\alpha +\delta }}^{\frac{\alpha }{2}}+\underset{u\in \lbrack t,T]}%
{\sup }\left\Vert \int_{u}^{T}\left\vert f^{i}(u,s,0,0)\right\vert
ds\right\Vert _{L_{G}^{\alpha +\delta }}^{\frac{\alpha }{2}}\text{ for }%
i=1,2.
\end{equation*}
\end{proposition}

\begin{proof}
Applying Proposition \ref{noYguji2} to (\ref{Ycha}), we obtain 
\begin{equation}
\begin{array}{ll}
\displaystyle\mathbb{\hat{E}}_{t}\left[ |\hat{Y}(s)|^{\alpha }\right] & %
\displaystyle\leq C\mathbb{\hat{E}}_{t}\left[ \mathbb{\hat{E}}_{s}\left[ |%
\hat{\phi}(s)|^{\alpha }+\left( \int_{s}^{T}\left\vert f^{1}(s,r,Y^{1}\left(
r\right) ,Z^{2}(s,r))-f^{2}(s,r,Y^{1}\left( r\right) ,Z^{2}(s,r))\right\vert
dr\right) ^{\alpha }\right] \right] .%
\end{array}%
\end{equation}%
Note that $\left\vert f^{1}(t,s,Y^{1}\left( s\right)
,Z^{2}(t,s))-f^{2}(t,s,Y^{2}\left( s\right) ,Z^{2}(t,s))\right\vert \leq |%
\hat{f}(t,s)|+L|\hat{Y}(s)|.$ Similar to the proof of Proposition \ref%
{Yguji1}, it is easy to prove (\ref{Yguji2_1}) and (\ref{Yguji2_2}).
\end{proof}

The proof of existence and uniqueness relies on a Picard iteration.
Therefore, it is necessary to verify that the following Picard iteration
sequence is well-defined. For each $n\in \mathbb{N}$, we define%
\begin{equation}
Y_{n}\left( t\right) =\phi \left( t\right) +\int_{t}^{T}f\left(
t,s,Y_{n-1}\left( s\right) ,Z_{n}\left( t,s\right) \right)
ds-\int_{t}^{T}Z_{n}\left( t,s\right) dB_{s}-K_{n}(t,T)  \label{picard}
\end{equation}%
and $Y_{0}(t)=0.$

\begin{lemma}
\label{uni_picard}Suppose that $\phi (\cdot )$ is $L_{G}^{\beta }$%
-continuous and $f$ satisfies (H1)-(H4) for some $\beta >1$. Then $G$-BSVIEs
(\ref{picard}) admit a unique solution $(Y_{n},Z_{n},K_{n})\in \mathfrak{%
\tilde{S}}_{G}^{\alpha }(\Delta \left( 0,T\right) )$ for any $1<\alpha
<\beta $ and $n\in \mathbb{N}.$
\end{lemma}

\begin{proof}
Assume that Eq.(\ref{picard}) admits a unique solution $%
(Y_{n-1},Z_{n-1},K_{n-1})\in \mathfrak{\tilde{S}}_{G}^{\alpha }(\Delta
\left( 0,T\right) )$ for each $n\in \mathbb{N}$ and $1<\alpha <\beta .$ Then
there exists a constant $\varepsilon _{0}>0$ with $\alpha +\varepsilon
_{0}<\beta $ such that $(Y_{n-1},Z_{n-1},K_{n-1})\in \mathfrak{\tilde{S}}%
_{G}^{\alpha +\varepsilon _{0}}(\Delta \left( 0,T\right) ).$ From (H1)-(H3)
and Theorem 4.7 in \cite{Hu2016Qc}, it is easy to check that for each $z\in 
\mathbb{R}
$ and $t\in \lbrack 0,T],$ $f(t,\cdot ,Y_{n-1}\left( \cdot \right) ,z)\in
M_{G}^{\alpha +\varepsilon _{0}}(t,T)$ and $\underset{t\in \lbrack 0,T]}{%
\sup }\mathbb{\hat{E}}\left[ \int_{t}^{T}\left\vert f(t,s,Y_{n-1}\left(
s\right) ,0)\right\vert ^{\alpha +\varepsilon _{0}}ds\right] <\infty .$
Moreover, for each $\left( t,s\right) ,\left( t^{\prime },s\right) \in
\Delta \left[ 0,T\right] $ and fixed $N\in 
\mathbb{N}
,$ we have%
\begin{equation}
\begin{array}{l}
\displaystyle\mathbb{\hat{E}}\Biggl[\underset{_{\{\left\vert z\right\vert
\leq N\}}}{\sup }\left\vert \int_{t\vee t^{\prime }}^{T}\left\vert
f(t^{\prime },s,Y_{n-1}\left( s\right) ,z)-f(t,s,Y_{n-1}\left( s\right)
,z)\right\vert ds\right\vert ^{\alpha +\varepsilon _{0}}\Biggr] \\ 
\displaystyle\leq C\left\{ \mathbb{\hat{E}}\Biggl[\underset{_{\{\left\vert
z\right\vert \leq N\}}}{\sup }\left\vert \int_{t\vee t^{\prime
}}^{T}\left\vert f(t^{\prime },s,Y_{n-1}\left( s\right) ,z)-f(t^{\prime
},s,Y_{n-1}^{N}\left( s\right) ,z)\right\vert ds\right\vert ^{\alpha
+\varepsilon _{0}}\Biggr]+\mathbb{\hat{E}}\Biggl[\underset{_{\{\left\vert
z\right\vert \leq N\}}}{\sup }\Biggl\vert\int_{t\vee t^{\prime
}}^{T}f(t^{\prime },s,Y_{n-1}^{N}\left( s\right) \Biggr.\Biggr.\right. \\ 
\text{ \ \ \ }\displaystyle\left. \Biggl.\Biggl.,z)-f(t,s,Y_{n-1}^{N}\left(
s\right) ,z)ds\Biggr\vert^{\alpha +\varepsilon _{0}}\Biggr]+\mathbb{\hat{E}}%
\Biggl[\underset{_{\{\left\vert z\right\vert \leq N\}}}{\sup }\left\vert
\int_{t\vee t^{\prime }}^{T}\left\vert f(t,s,Y_{n-1}^{N}\left( s\right)
,z)-f(t,s,Y_{n-1}\left( s\right) ,z)\right\vert ds\right\vert ^{\alpha
+\varepsilon _{0}}\Biggr]\right\} \\ 
\displaystyle\leq C\Biggl\{\mathbb{\hat{E}}\Biggl[\int_{0}^{T}\left\vert
Y_{n-1}\left( s\right) \right\vert ^{\alpha +\varepsilon _{0}}I_{\left\{
\left\vert Y_{n-1}\left( s\right) \right\vert \geq N\right\} }ds\Biggr]+%
\mathbb{\hat{E}}\Biggl[\underset{_{\{\left\vert y\right\vert +\left\vert
z\right\vert \leq 2N\}}}{\sup }\left\vert \int_{t\vee t^{\prime
}}^{T}\left\vert f(t^{\prime },s,y,z)-f(t,s,y,z)\right\vert ds\right\vert
^{\alpha +\varepsilon _{0}}\Biggr]\Biggr\},%
\end{array}
\label{Y_(H4)}
\end{equation}%
where $Y_{n-1}^{N}\left( s\right) =Y_{n-1}\left( s\right) I_{\left\{
\left\vert Y_{n-1}\left( s\right) \right\vert \leq N\right\} }$ and the
constant $C$ depends on $\alpha ,\varepsilon _{0},T,L$. Applying the similar
argument as in Lemma \ref{contin_para} and Theorem \ref{dengjia}, we deduce
from (\ref{Y_(H4)}) that $(Y_{n},Z_{n},K_{n})\in \mathfrak{\tilde{S}}%
_{G}^{\alpha }(\Delta \left( 0,T\right) )$ for each $n\in \mathbb{N}$ and $%
1<\alpha <\beta .$
\end{proof}

In the following, we establish the well-posedness of the $G$-BSVIE (\ref%
{G-BSVIE2}), which is the main result of this subsection.

\begin{theorem}
\label{Y_unique}Suppose that $\phi (\cdot )$ is $L_{G}^{\beta }$-continuous
and $f$ satisfies (H1)-(H4) for some $\beta >1$. Then $G$-BSVIE (\ref%
{G-BSVIE2}) admits a unique solution $(Y,Z,K)\in \mathfrak{\tilde{S}}%
_{G}^{\alpha }(\Delta \left( 0,T\right) )$ for any $1<\alpha <\beta .$
Moreover, $Y(\cdot ),K(\cdot ,T)$ are $L_{G}^{\alpha }$-continuous and $%
Z\left( \cdot ,\cdot \right) $ is $H_{G}^{\alpha }$-continuous$.$
\end{theorem}

\begin{proof}
\textbf{Uniqueness.} Let $\left( Y_{1},Z_{1},K_{1}\right) ,$ $\left(
Y_{2},Z_{2},K_{2}\right) $ be two solutions of $G$-BSVIE (\ref{G-BSVIE2}).
Since $Y_{i},$ $i=1,2$ are known processes, we get from Proposition \ref%
{noYguji2} that%
\begin{equation*}
\begin{array}{l}
\displaystyle\mathbb{\hat{E}}\left[ \left\vert Y_{1}\left( t\right)
-Y_{2}\left( t\right) \right\vert ^{\alpha }\right] \\ 
\displaystyle\leq C\mathbb{\hat{E}}\left[ \left( \int_{t}^{T}\left\vert
f(t,s,Y_{1}\left( s\right) ,Z_{2}(t,s))-f(t,s,Y_{2}\left( s\right)
,Z_{2}(t,s))\right\vert ds\right) ^{\alpha }\right] \\ 
\displaystyle\leq C\int_{t}^{T}\mathbb{\hat{E}}\left[ \left\vert Y_{1}\left(
s\right) -Y_{2}\left( s\right) \right\vert ^{\alpha }\right] ds,%
\end{array}%
\end{equation*}%
where $C$ depends on $\alpha ,T,\underline{\sigma },L.$ Applying Gronwall's
inequality yields that $\left\Vert Y_{1}-Y_{2}\right\Vert _{L_{G}^{\alpha
}}=0.$ Then we have $\left\Vert Y_{1}-Y_{2}\right\Vert _{M_{G}^{\alpha }}=0,$
which implies that $Y_{1}=Y_{2}.$ Furthermore, by Proposition \ref{Yguji2},
we get $Z_{1}=Z_{2},$ $K_{1}=K_{2}$.\newline
\textbf{Existence.} Let $\delta $ be an undetermined constant satisfying $%
0<\delta \leq T.$ By using a different backward iteration procedure, we
prove the existence of the solution to $G$-BSVIE (\ref{G-BSVIE2}) on each
local interval.

Part 1: $\left( T-\delta ,T\right] $. Define picard sequences on $\left(
T-\delta ,T\right] $ as follows:%
\begin{equation}
Y_{n}^{T-\delta }\left( t\right) =\phi \left( t\right) +\int_{t}^{T}f\left(
t,s,Y_{n-1}^{T-\delta }\left( s\right) ,Z_{n}^{T-\delta }\left( t,s\right)
\right) ds-\int_{t}^{T}Z_{n}^{T-\delta }\left( t,s\right) dB_{s}-K_{n}(t,T),%
\text{ }n\in \mathbb{N}  \label{picard1}
\end{equation}%
and $Y_{0}^{T-\delta }\left( t\right) =0.$ From Lemma \ref{uni_picard}$,$ (%
\ref{picard1}) admits a unique solution $(Y_{n}^{T-\delta },Z_{n}^{T-\delta
},K_{n}^{T-\delta })\in \mathfrak{\tilde{S}}_{G}^{\alpha }(\Delta \left(
T-\delta ,T\right) )$ for each $n\in \mathbb{N}$ and $1<\alpha <\beta .$
Thus, it follows from the estimates in Proposition \ref{noYguji2} that%
\begin{equation}
\begin{array}{l}
\displaystyle\mathbb{\hat{E}}\left[ \int_{T-\delta }^{T}\left\vert
Y_{n}^{T-\delta }(t)-Y_{n-1}^{T-\delta }(t)\right\vert ^{\alpha }dt\right] 
\\ 
\displaystyle\leq C\int_{T-\delta }^{T}\mathbb{\hat{E}}\left[ \left(
\int_{T-\delta }^{T}\left\vert f\left( t,s,Y_{n-1}^{T-\delta }\left(
s\right) ,Z_{n-2}^{T-\delta }\left( t,s\right) \right) -f\left(
t,s,Y_{n-2}^{T-\delta }\left( s\right) ,Z_{n-2}^{T-\delta }\left( t,s\right)
\right) \right\vert ds\right) ^{\alpha }\right] dt \\ 
\displaystyle\leq CL^{\alpha }\delta ^{\alpha }\mathbb{\hat{E}}\left[
\int_{T-\delta }^{T}\left\vert Y_{n-1}^{T-\delta }(t)-Y_{n-2}^{T-\delta
}(t)\right\vert ^{\alpha }dt\right] ,%
\end{array}
\label{T0process}
\end{equation}%
where the constant $C=C\left( \alpha ,T,L,\underline{\sigma }\right) .$ Set $%
\delta =\frac{1}{\sqrt[\alpha ]{2C}L},\ $it holds that there exists a
constant $k=\left[ \frac{T}{\delta }\right] \in \mathbb{N}$ such that $%
k\delta \geq T.$ From (\ref{T0process}) and Proposition \ref{noYguji1}, we
have%
\begin{equation}
\mathbb{\hat{E}}\left[ \int_{T-\delta }^{T}\left\vert Y_{n}^{T-\delta
}(t)-Y_{n-1}^{T-\delta }(t)\right\vert ^{\alpha }dt\right] \leq \frac{1}{2}%
\mathbb{\hat{E}}\left[ \int_{T-\delta }^{T}\left\vert Y_{n-1}^{T-\delta
}\left( t\right) -Y_{n-2}^{T-\delta }\left( t\right) \right\vert ^{\alpha }dt%
\right] \leq \frac{1}{2^{n-1}}\mathbb{\hat{E}}\left[ \int_{T-\delta
}^{T}\left\vert Y_{1}^{T-\delta }\left( t\right) \right\vert ^{\alpha }dt%
\right]   \label{T1process}
\end{equation}%
and 
\begin{equation*}
\mathbb{\hat{E}}\left[ \int_{T-\delta }^{T}\left\vert Y_{1}^{T-\delta
}\left( t\right) \right\vert ^{\alpha }dt\right] \leq C\underset{t\in \left[
0,T\right] }{\sup }\mathbb{\hat{E}}\left[ \left\vert \phi \left( t\right)
\right\vert ^{\alpha }+\left( \int_{t}^{T}\left\vert f\left( t,s,0,0\right)
\right\vert ds\right) ^{\alpha }\right] \leq C,
\end{equation*}%
where $C$ depends on $\alpha ,T,\underline{\sigma },L.$ Letting $%
n\rightarrow \infty ,$ we deduce that $\left\{ Y_{n}^{T-\delta }\right\}
_{n\in \mathbb{N}}$ is a Cauchy sequence in $M_{G}^{\alpha }\left( T-\delta
,T\right) .$ In a similar manner, we can easily verify that $\left\{
Y_{n}^{T-\delta }\left( t\right) \right\} _{n\in \mathbb{N}}$ is also a
Cauchy sequence in $L_{G}^{\alpha }\left( \Omega _{t}\right) $ for each $%
t\in \left( T-\delta ,T\right] .$ Therefore, there exist a process $Y(\cdot )
$ such that 
\begin{equation*}
\underset{n\rightarrow \infty }{\lim }\mathbb{\hat{E}}\left[ \int_{T-\delta
}^{T}\left\vert Y_{n}^{T-\delta }(t)-Y(t)\right\vert ^{\alpha }dt\right] =0%
\text{ and }\underset{n\rightarrow \infty }{\lim }\mathbb{\hat{E}}\left[
\left\vert Y_{n}^{T-\delta }\left( t\right) -Y(t)\right\vert ^{\alpha }%
\right] =0\text{, for each }t\in \left( T-\delta ,T\right] ,
\end{equation*}%
which implies that $G$-BSVIE (\ref{G-BSVIE2}) admits a solution $Y\in \tilde{%
M}_{G}^{\alpha }\left( T-\delta ,T\right) .$ It follows that $Y$ can be
regarded as a known process. Then by Theorem \ref{dengjia}, $G$-BSVIE (\ref%
{G-BSVIE2}) has a unique solution $(Y,Z,K)\in \mathfrak{\tilde{S}}%
_{G}^{\alpha }(\Delta \left( T-\delta ,T\right) ).$ Hence, we obtain the
values $(Y\left( t\right) I_{\left( T-\delta ,T\right] }\left( t\right)
,Z\left( t,s\right) I_{\Delta \left( T-\delta ,T\right] }\left( t,s\right)
,K\left( t,s\right) $$I_{\Delta \left( T-\delta ,T\right] }\left( t,s\right)
).$ Moreover, by Propositon \ref{Yguji1}, we derive 
\begin{equation*}
\left\Vert Y\right\Vert _{M_{G}^{\alpha }}\leq C\left( \alpha ,T,\underline{%
\sigma },L\right) \text{ and }\left\Vert Y\right\Vert _{L_{G}^{\alpha }}\leq
C\left( \alpha ,T,\underline{\sigma },L\right) .
\end{equation*}

Part 2: $\left( T-2\delta ,T-\delta \right] .$ Based on the value $Y(\cdot )$
on $\left( T-\delta ,T\right] ,$ we construct the following picard sequence$.
$ Define%
\begin{equation}
\begin{array}{ll}
\displaystyle Y_{n}^{T-2\delta }\left( t\right) = & \displaystyle\phi \left(
t\right) +\int_{t}^{T}f\left( t,s,Y(s)I_{(T-\delta ,T]}\left( s\right)
+Y_{n-1}^{T-2\delta }\left( s\right) I_{\left[ t,T-\delta \right] }\left(
s\right) ,Z_{n}^{T-2\delta }\left( t,s\right) \right) ds \\ 
& \displaystyle-\int_{t}^{T}Z_{n}^{T-2\delta }\left( t,s\right)
dB_{s}-K_{n}^{T-2\delta }(t,T),\text{ }n\in \mathbb{N}%
\end{array}
\label{picard2}
\end{equation}%
and $Y_{0}^{T-2\delta }\left( t\right) =0.$ Similarly, by Lemma \ref%
{uni_picard}, the $G$-BSVIEs (\ref{picard2}) admit a unique solution $%
(Y_{n}^{T-2\delta },Z_{n}^{T-2\delta },$\newline
$K_{n}^{T-2\delta })\in \mathfrak{\tilde{S}}_{G}^{\alpha }(\Delta \left(
T-2\delta ,T-\delta \right) )$ for each $n\in \mathbb{N}$. Following the
similar proof of (\ref{T1process}), we can obtain from Proposition \ref%
{noYguji1} and Proposition \ref{noYguji2} that%
\begin{equation}
\begin{array}{l}
\displaystyle\mathbb{\hat{E}}\left[ \int_{T-2\delta }^{T-\delta }\left\vert
Y_{n}^{T-2\delta }(t)-Y_{n-1}^{T-2\delta }(t)\right\vert ^{\alpha }dt\right] 
\\ 
\displaystyle\leq C\Biggl\{\int_{T-2\delta }^{T-\delta }\mathbb{\hat{E}}%
\Biggl[\Biggl(\int_{T-2\delta }^{T}\Bigl\vert f\Bigl(t,s,Y(s)I_{(T-\delta
,T]}(s)+Y_{n-1}^{T-2\delta }(s)I_{[t,T-\delta ]}(s),Z_{n-1}^{T-2\delta }(t,s)%
\Bigr) \\ 
\text{ \ }\displaystyle-f\Bigl(t,s,Y(s)I_{(T-\delta
,T]}(s)+Y_{n-2}^{T-2\delta }(s)I_{[t,T-\delta ]}(s),Z_{n-1}^{T-2\delta }(t,s)%
\Bigr)\Bigr\vert ds\Biggr)^{\alpha }\Biggr]dt\Biggr\} \\ 
\displaystyle\leq \frac{1}{2}\mathbb{\hat{E}}\left[ \int_{T-2\delta
}^{T-\delta }\left\vert Y_{n-1}^{T-2\delta }\left( t\right)
-Y_{n-2}^{T-2\delta }\left( t\right) \right\vert ^{\alpha }dt\right] \leq 
\frac{1}{2^{n-1}}\mathbb{\hat{E}}\left[ \int_{T-2\delta }^{T-\delta
}\left\vert Y_{1}^{T-2\delta }\left( t\right) \right\vert ^{\alpha }dt\right]
\end{array}
\label{Y_unique_T2}
\end{equation}%
and%
\begin{equation*}
\begin{array}{l}
\displaystyle\mathbb{\hat{E}}\left[ \int_{T-2\delta }^{T-\delta }\left\vert
Y_{1}^{T-2\delta }\left( t\right) \right\vert ^{\alpha }dt\right]  \\ 
\displaystyle\leq \int_{T-2\delta }^{T-\delta }\mathbb{\hat{E}}\left[
\left\vert \phi \left( t\right) \right\vert ^{\alpha }+\left(
\int_{t}^{T}\left\vert f\left( t,s,Y(s)I_{\left[ T-\delta ,T\right] }\left(
s\right) ,0\right) \right\vert ds\right) ^{\alpha }\right] dt \\ 
\displaystyle\leq C\left\{ \underset{t\in \left[ 0,T\right] }{\sup }\mathbb{%
\hat{E}}\left[ \left\vert \phi \left( t\right) \right\vert ^{\alpha }+\left(
\int_{t}^{T}\left\vert f\left( t,s,0,0\right) \right\vert ds\right) ^{\alpha
}\right] +\mathbb{\hat{E}}\left[ \int_{T-\delta }^{T}\left\vert
Y(s)\right\vert ^{\alpha }ds\right] \right\} \leq C,%
\end{array}%
\end{equation*}%
\newline
where $C$ is a constant depending on $\alpha ,T,\underline{\sigma },L.$
Similarly, there exists a process $Y$ such that $Y_{n}^{T-2\delta }$
converges to $Y$ in $\tilde{M}_{G}^{\alpha }\left( T-2\delta ,T-\delta
\right) .$ By Theorem \ref{dengjia}, $G$-BSVIE (\ref{G-BSVIE2}) has a
solution $\left( Y,Z,K\right) \in \mathfrak{\tilde{S}}_{G}^{\alpha }(\Delta
\left( T-2\delta ,T-\delta \right) )$. Then, the values of $(Y\left(
t\right) I_{\left( T-2\delta ,T-\delta \right] }\left( t\right) ,Z\left(
t,s\right) I_{\Delta \left( T-2\delta ,T-\delta \right] }\left( t,s\right)
,K\left( t,s\right) $$I_{\Delta \left( T-2\delta ,T-\delta \right] }\left(
t,s\right) )$ are derived.

The rest of the proof of existence is similar to the case when $t\in \left(
T-2\delta ,T-\delta \right] .$ Arguing by induction, we obtian that $G$%
-BSVIE (\ref{G-BSVIE2}) admits a solution $\left( Y,Z,K\right) \in \mathfrak{%
\tilde{S}}_{G}^{\alpha }(\Delta \left( T-\left( i+1\right) \delta ,T-i\delta
\right) )$ for each $0\leq i\leq k-1.\ $Set 
\begin{equation}
Y\left( t\right) =\overset{k-2}{\underset{i=0}{\sum }}Y\left( t\right)
I_{\left( T-(i+1)\delta ,T-i\delta \right] }\left( t\right) +Y\left(
t\right) I_{\left[ 0,T-\left( k-1\right) \delta \right] }\left( t\right) ,
\label{Y_exist}
\end{equation}%
\begin{equation}
Z\left( t,s\right) =\overset{k-2}{\underset{i=0}{\sum }}Z\left( t,s\right)
I_{\Delta \left( T-(i+1)\delta ,T-i\delta \right] }\left( t,s\right)
+Z\left( t,s\right) I_{\Delta \left[ 0,T-\left( k-1\right) \delta \right]
}\left( t,s\right)  \label{Z_exist}
\end{equation}%
and 
\begin{equation}
K\left( t,s\right) =\overset{k-2}{\underset{i=0}{\sum }}K\left( t,s\right)
I_{\Delta \left( T-(i+1)\delta ,T-i\delta \right] }\left( t,s\right)
+K\left( t,s\right) _{\Delta \left[ 0,T-\left( k-1\right) \delta \right]
}\left( t,s\right) .  \label{K_exist}
\end{equation}%
Then we prove the existence of the solution to the $G$-BSVIE (\ref{G-BSVIE2}%
) on entire interval $\left[ 0,T\right] $.

\textbf{Continuity. }Let $\left( Y,Z,K\right) $ be the unique adapted
solution of $G$-BSVIE (\ref{G-BSVIE2}) (see (\ref{Y_exist})-(\ref{K_exist}%
)). Define the following $G$-BSDEs parameterized by each fixed $t\in \lbrack
0,T]$:%
\begin{equation}
\lambda (t,r)=\phi (t)+\int_{r}^{T}f(t,s,Y\left( s\right) ,\mu
(t,s))ds-\int_{r}^{T}\mu (t,s)dB_{s}-\left( \eta (t,T)-\eta (t,r)\right)
,\quad (t,r)\in \Delta \lbrack 0,T].  \label{paraY}
\end{equation}%
By Theorem \ref{Djie}, we derive that (\ref{paraY}) has a unique solution $%
(\lambda (t,\cdot ),\mu (t,\cdot ),\eta (t,\cdot ))\in \mathfrak{S}%
_{G}^{\alpha }(t,T)$ and $\eta (t,t)=0$ for each fixed $t\in \lbrack 0,T].$
Similar to the analysis as in Theorem \ref{dengjia} and (\ref{Y_(H4)}), we
derive 
\begin{equation*}
\left( Y(t),Z(t,s),K(t,s)\right) =\left( \lambda (t,t),\mu (t,s),\eta
(t,s)\right) ,\text{ for each }\left( t,s\right) \in \Delta \lbrack 0,T].
\end{equation*}%
Moreover, the continuity of $\lambda (\cdot ,\cdot )$ and $\mu \left( \cdot
,\cdot \right) $ can be obtained from Lemma \ref{contin_para} similarly.
Therefore, following the proof of (\ref{noYcontin})-(\ref{noYK_contin}), we
get the continuity of $Y(\cdot )$, $Z\left( \cdot ,\cdot \right) $ and $%
K\left( \cdot ,T\right) .$
\end{proof}

\section{Comparison of Backward Stochastic Volterra integral equations
driven by $G$-Brownian motion}

\noindent

In this section, we establish the comparison theorem for $G$-BSVIE (\ref%
{G-BSVIE2}). Following the idea of Section 3, we shall first consider the
comparison theorem for $G$-BSVIE (\ref{noY}), where the generator $f$ is
independent of $y$.

\begin{proposition}
\label{compar_noY}Let $(Y^{i},Z^{i},K^{i})\in \mathfrak{\tilde{S}}%
_{G}^{\alpha }(\Delta \left( 0,T\right) )$ be the solution of $G$-BSVIE of
the form:%
\begin{equation}
Y^{i}(t)=\phi
^{i}(t)+\int_{t}^{T}f^{i}(t,s,Z^{i}(t,s))ds-%
\int_{t}^{T}Z^{i}(t,s)dB_{s}-K^{i}(t,T),\quad t\in \lbrack 0,T],
\label{compar_noY1}
\end{equation}%
where $\phi ^{i}(\cdot )$ is $L_{G}^{\beta }$-continuous and $f^{i}$
satisfies (H1)-(H4) for any $1<\alpha <\beta $ and $i=1,2$. If $\phi
^{1}\geq \phi ^{2}$ and $f^{1}\geq f^{2},$ then $Y^{1}\geq Y^{2}$.
\end{proposition}

\begin{proof}
First of all$,$ we consider the following $G$-BSDEs parameterized by each
fixed $t\in \lbrack 0,T]$ for $i=1,2:$%
\begin{equation*}
\lambda ^{i}(t,r)=\phi ^{i}(t)+\int_{r}^{T}f^{i}(t,s,\mu
^{i}(t,s))ds-\int_{r}^{T}\mu ^{i}(t,s)dB_{s}-\left( \eta ^{i}(t,T)-\eta
^{i}(t,r)\right) ,\quad (t,r)\in \Delta \left[ 0,T\right] .
\end{equation*}%
By Theorem \ref{Dcompar} we have $\lambda ^{1}(t,r)\geq \lambda ^{2}(t,r),$ $%
r\in \lbrack t,T].$ Recalling Theorem \ref{dengjia}, since we have%
\begin{equation*}
Y^{i}(t)=\lambda ^{i}(t,t),\text{ for each }t\in \left[ 0,T\right] \text{
and }i=1,2,
\end{equation*}%
it follows that $Y^{1}\left( t\right) \geq Y^{2}\left( t\right) $ by taking $%
r=t$.
\end{proof}

Next, we turn to the comparison theorem for $G$-BSVIE (\ref{G-BSVIE2}), and
the main results are presented as follows.

\begin{theorem}
\label{compar_Y_incre}Let $(Y^{i},Z^{i},K^{i})\in \mathfrak{\tilde{S}}%
_{G}^{\alpha }(\Delta \left( 0,T\right) )$ be the solution of $G$-BSVIE of
the form: 
\begin{equation}
Y^{i}(t)=\phi
^{i}(t)+\int_{t}^{T}f^{i}(t,s,Y^{i}(s),Z^{i}(t,s))ds-%
\int_{t}^{T}Z^{i}(t,s)dB_{s}-K^{i}(t,T),\quad t\in \lbrack 0,T],
\label{compar_Y_incre2}
\end{equation}%
where $\phi ^{i}(\cdot )$ is $L_{G}^{\beta }$-continuous and $f^{i}$
satisfies (H1)-(H4) for $i=1,2$ with any $1<\alpha <\beta .$ Moreover, there
exists $i=1$ or $2$ such that 
\begin{equation}
f^{i}\left( t,s,y_{1},z\right) \geq f^{i}\left( t,s,y_{2},z\right) \text{
for each }y_{1}\geq y_{2},\text{ }\left( t,s,y_{1},z\right) ,\left(
t,s,y_{2},z\right) \in \Delta \left[ 0,T\right] \times \mathbb{R}^{2}.
\label{compar_fi1}
\end{equation}%
If $\phi ^{1}\geq \phi ^{2}$ and 
\begin{equation}
f^{1}\left( t,s,y,z\right) \geq f^{2}\left( t,s,y,z\right) \text{ for each }%
\left( t,s,y,z\right) \in \Delta \left[ 0,T\right] \times \mathbb{R}^{2},
\label{compar_fi2}
\end{equation}%
then we have $Y^{1}\left( t\right) \geq Y^{2}\left( t\right) ,$ q.s., $t\in %
\left[ 0,T\right] .$

\begin{proof}
Without loss of generality, we assume that condition (\ref{compar_fi1})
holds for $i=1.$ Let $\delta $ be an undetermined constant satisfying $%
0<\delta \leq T.$ Next, we prove the comparison theorem for $G$-BSVIE (\ref%
{G-BSVIE2}) by induction on each local interval.

Part 1: $\left( T-\delta ,T\right] .$ Following the idea of Theorem \ref%
{Y_unique}, we define the $G$-BSVIEs for each $t\in \left( T-\delta ,T\right]
$, 
\begin{equation}
\bar{Y}_{n}^{T-\delta }(t)=\phi ^{1}(t)+\int_{t}^{T}f^{1}(t,s,\bar{Y}%
_{n-1}^{T-\delta }(s),\bar{Z}_{n}^{T-\delta }(t,s))ds-\int_{t}^{T}\bar{Z}%
_{n}^{T-\delta }(t,s)dB_{s}-\bar{K}_{n}^{T-\delta }(t,T)  \label{compar_T1}
\end{equation}%
and $\bar{Y}_{0}^{T-\delta }(t)=Y^{2}(t),$ which together with Lemma \ref%
{uni_picard} implies that there admits a unique solution $\left( \bar{Y}%
_{n}^{T-\delta },\bar{Z}_{n}^{T-\delta },\bar{K}_{n}^{T-\delta }\right) \in 
\mathfrak{\tilde{S}}_{G}^{\alpha }(\Delta \left( T-\delta ,T\right) )$ for
each $n\in \mathbb{N}.$ For $n=1$ in (\ref{compar_T1}), since we have $\phi
^{1}\geq \phi ^{2}$ and 
\begin{equation*}
f^{1}(t,s,\bar{Y}_{0}^{T-\delta }(s),z)=f^{1}(t,s,Y^{2}(s),z)\geq
f^{2}(t,s,Y^{2}(s),z),\text{ for each }\left( t,s\right) \in \Delta \left(
T-\delta ,T\right] ,z\in \mathbb{R},
\end{equation*}%
we derive from Proposition \ref{compar_noY} that 
\begin{equation*}
\bar{Y}_{1}^{T-\delta }(t)\geq Y^{2}\left( t\right) =\bar{Y}_{0}^{T-\delta
}(t),\text{ for each }t\in \left( T-\delta ,T\right] .
\end{equation*}%
For $n=2$ in (\ref{compar_T1})$,$ since $f^{1}$ is non-decreasing in $y,$ we
have 
\begin{equation*}
f^{1}(t,s,\bar{Y}_{1}^{T-\delta }(s),z)\geq f^{1}(t,s,\bar{Y}_{0}^{T-\delta
}\left( s\right) ,z),\text{ for each }\left( t,s\right) \in \Delta \left(
T-\delta ,T\right] ,z\in \mathbb{R}.
\end{equation*}%
Similarly, we obtain from Proposition \ref{compar_noY} that $\bar{Y}%
_{2}^{T-\delta }(t)\geq \bar{Y}_{1}^{T-\delta }\left( t\right) ,$ $t\in
\left( T-\delta ,T\right] .$ Repeating this procedure, we deduce that for
each $n\in \mathbb{N},$ 
\begin{equation}
\bar{Y}_{n}^{T-\delta }(t)\geq \cdot \cdot \cdot \geq \bar{Y}_{2}^{T-\delta
}(t)\geq \bar{Y}_{1}^{T-\delta }(t)\geq Y^{2}\left( t\right) ,\text{ }t\in
\left( T-\delta ,T\right] .  \label{compar_T1n}
\end{equation}%
Consequently, it suffices to prove $\bar{Y}_{n}^{T-\delta }$ converges to $%
Y_{1}$ in $\tilde{M}_{G}^{\alpha }(T-\delta ,T).$ Applying Proposition \ref%
{noYguji2} yields that%
\begin{equation}
\begin{array}{l}
\displaystyle\mathbb{\hat{E}}\left[ \int_{T-\delta }^{T}\left\vert \bar{Y}%
_{n}^{T-\delta }(t)-\bar{Y}_{n-1}^{T-\delta }(t)\right\vert ^{\alpha }dt%
\right]  \\ 
\displaystyle\leq C\int_{T-\delta }^{T}\mathbb{\hat{E}}\left[ \left(
\int_{t}^{T}\left\vert f^{1}(t,s,\bar{Y}_{n-1}^{T-\delta }(s),\bar{Z}%
_{n-2}^{T-\delta }(t,s))-f^{1}(t,s,\bar{Y}_{n-2}^{T-\delta }(s),\bar{Z}%
_{n-2}^{T-\delta }(t,s))\right\vert ds\right) ^{\alpha }\right] dt \\ 
\displaystyle\leq CL^{\alpha }\delta ^{\alpha }\mathbb{\hat{E}}\left[
\int_{T-\delta }^{T}\left\vert \bar{Y}_{n-1}^{T-\delta }(t)-\bar{Y}%
_{n-2}^{T-\delta }(t)\right\vert ^{\alpha }dt\right] ,%
\end{array}
\label{compar_1}
\end{equation}%
where the constant $C$ depends on $\alpha ,T,L,\underline{\sigma }.$ Set $%
\delta =\frac{1}{\sqrt[\alpha ]{2C}L},\ $then there exists a constant $k=%
\left[ \frac{T}{\delta }\right] \in \mathbb{N}$ such that $k\delta \geq T,$
where constant $C=C\left( \alpha ,T,L,\underline{\sigma }\right) $ is taken
in (\ref{compar_1}). Then we know that 
\begin{equation*}
\mathbb{\hat{E}}\left[ \int_{T-\delta }^{T}\left\vert \bar{Y}_{n}^{T-\delta
}(t)-\bar{Y}_{n-1}^{T-\delta }(t)\right\vert ^{\alpha }dt\right] \leq \frac{1%
}{2}\mathbb{\hat{E}}\left[ \int_{T-\delta }^{T}\left\vert \bar{Y}%
_{n-1}^{T-\delta }(t)-\bar{Y}_{n-2}^{T-\delta }(t)\right\vert ^{\alpha }dt%
\right] \leq \frac{1}{2^{n}}\mathbb{\hat{E}}\left[ \int_{T-\delta
}^{T}\left\vert Y^{2}\left( t\right) \right\vert ^{\alpha }dt\right] ,
\end{equation*}%
which implies that $\left\{ \bar{Y}_{n}^{T-\delta }\right\} _{n\in \mathbb{N}%
}$ is a cauchy sequence in $\tilde{M}_{G}^{\alpha }(T-\delta ,T).$ Based on
the uniqueness of the solutions to $G$-BSVIEs (\ref{compar_Y_incre2}) for $%
i=1,$ we get 
\begin{equation*}
\underset{n\rightarrow \infty }{\lim }\mathbb{\hat{E}}\left[ \int_{T-\delta
}^{T}\left\vert \bar{Y}_{n}^{T-\delta }(t)-Y^{1}(t)\right\vert ^{\alpha }dt%
\right] =0.
\end{equation*}%
Combining with (\ref{compar_T1n}), we see that
\begin{equation}
Y^{1}\left( t\right) \geq Y^{2}\left( t\right) \text{, q.s., }t\in \left(
T-\delta ,T\right] .  \label{T}
\end{equation}

Part 2: $\left( T-2\delta ,T-\delta \right] .$ Next, we define the following
sequences of $G$-BSVIEs for each $n\in \mathbb{N}:$%
\begin{equation}
\begin{array}{ll}
\displaystyle\bar{Y}_{n}^{T-2\delta }(t)= & \displaystyle\phi
^{1}(t)+\int_{t}^{T}f^{1}(t,s,\bar{Y}_{n-1}^{T-2\delta }(s)I_{[t,T-\delta
]}(s)+Y^{1}\left( s\right) I_{(T-\delta ,T]}(s),\bar{Z}_{n-1}^{T-2\delta
}(t,s))ds \\ 
& \displaystyle-\int_{t}^{T}\bar{Z}_{n-1}^{T-2\delta }(t,s)dB_{s}-\bar{K}%
_{n-1}^{T-2\delta }(t,T),\text{ }t\in \left( T-2\delta ,T-\delta \right] ,%
\end{array}
\label{compar_T2}
\end{equation}%
and $\bar{Y}_{0}^{T-2\delta }(t)=Y^{2}\left( t\right) ,$ which admits a
unique solution $\left( \bar{Y}_{n}^{T-2\delta },\bar{Z}_{n}^{T-2\delta },%
\bar{K}_{n}^{T-2\delta }\right) \in \mathfrak{\tilde{S}}_{G}^{\alpha
}(\Delta \left( T-2\delta ,T-\delta \right) )$ by Lemma \ref{uni_picard}.
For $n=1$ in (\ref{compar_T2})$,$ since (\ref{T}) holds, it follows from (%
\ref{compar_fi1}) and (\ref{compar_fi2}) that for each $\left( t,s\right)
\in \Delta \left( T-2\delta ,T-\delta \right] ,z\in \mathbb{R},$ 
\begin{equation*}
\displaystyle f^{1}(t,s,\bar{Y}_{0}^{T-2\delta }(s)I_{[t,T-\delta
]}(s)+Y^{1}\left( s\right) I_{(T-\delta ,T]}(s),z)\displaystyle\geq
f^{1}(t,s,Y^{2}\left( s\right) ,z)\geq f^{2}(t,s,Y^{2}(s),z),\text{ q.s.}
\end{equation*}%
Then by Proposition \ref{compar_noY} and $\phi ^{1}\geq \phi ^{2}$, we
derive 
\begin{equation*}
\bar{Y}_{1}^{T-2\delta }(t)\geq Y^{2}\left( t\right) =\bar{Y}_{0}^{T-2\delta
}(t),\text{ q.s., }t\in \left( T-2\delta ,T-\delta \right] .
\end{equation*}
For $n=2$ in (\ref{compar_T2})$,$ similar to the above analysis, we deduce
from condition (\ref{compar_fi1}) that 
\begin{equation*}
f^{1}(t,s,\bar{Y}_{1}^{T-2\delta }(s)I_{[t,T-\delta ]}(s)+Y^{1}\left(
s\right) I_{(T-\delta ,T]}(s),z)\geq f^{1}(t,s,\bar{Y}_{0}^{T-2\delta
}(s)I_{[t,T-\delta ]}(s)+Y^{1}\left( s\right) I_{(T-\delta ,T]}(s),z),\text{
q.s.}
\end{equation*}%
From Proposition \ref{compar_noY}, we have%
\begin{equation*}
\bar{Y}_{2}^{T-2\delta }(t)\geq \bar{Y}_{1}^{T-2\delta }(t),\text{ q.s., }%
t\in \left( T-2\delta ,T-\delta \right] .
\end{equation*}
Then, repeating this procedure, we get 
\begin{equation}
\bar{Y}_{n}(t)\geq \cdot \cdot \cdot \geq \bar{Y}_{2}^{T-2\delta }(t)\geq 
\bar{Y}_{1}^{T-2\delta }(t)\geq \bar{Y}_{0}^{T-2\delta }(t)=Y^{2}\left(
t\right) ,\text{ q.s., }t\in \left( T-2\delta ,T-\delta \right] .
\label{compar_T2n}
\end{equation}%
Moreover, by similar analysis\ as in (\ref{Y_unique_T2}), we deduce that $%
\bar{Y}_{n}^{T-2\delta }\ $converges to $Y^{1}$ in $\tilde{M}_{G}^{\alpha
}(T-2\delta ,T-\delta ).$ Combining with (\ref{compar_T2n}), we derive $%
Y^{1}\left( t\right) \geq Y^{2}\left( t\right) ,$ q.s., $t\in \left(
T-2\delta ,T-\delta \right] .$

Arguing by induction, since 
\begin{equation*}
Y^{1}\left( t\right) =\overset{k-2}{\underset{i=0}{\sum }}Y^{1}\left(
t\right) I_{\left( T-(i+1)\delta ,T-i\delta \right] }\left( t\right)
+Y^{1}\left( t\right) I_{\left[ 0,T-\left( k-1\right) \delta \right] }\left(
t\right) ,
\end{equation*}%
we conclude that 
$Y^{1}\left( t\right) \geq Y^{2}\left( t\right) \text{, q.s., }t\in \left[ 0,T%
\right] $.
\end{proof}
\end{theorem}

\end{document}